\documentclass[11pt]{amsart}
\usepackage{amssymb}
\usepackage{amsmath}
\usepackage{amsthm}
\usepackage{braket}
\usepackage{pdfpages}
\usepackage{graphicx}
\usepackage{caption}
\usepackage{subcaption}

\theoremstyle{plain}
\newtheorem{theorem}{Theorem}[section]

\theoremstyle{theorem}
\newtheorem{prop}[theorem]{Proposition}
\newtheorem{lem}[theorem]{Lemma}

\newtheorem{cl}{Claim}[]

\theoremstyle{definition}
\newtheorem{defn}[theorem]{Definition}
\newtheorem{rmk}[theorem]{Remark}

\newcommand{\m}{M\"{o}bius }
\newcommand{\R}{{\mathbb R}}

\numberwithin{figure}{section}

\title{Colorings of  the $n$-Sphere and Inversive Geometry}		
\author{Joel C. Gibbons}
\address{Logistic Research \& Trading Co., P.O. Box 63, St. Joseph, MI 49085}
\author{ Yusheng Luo}
\address{Dept. of Mathematics, National Univ. of Singapore, 10 Lower Kent Ridge Road, Singapore 119076}
\date{July 19, 2013, v.arXiv}

\begin{document}

\begin{abstract}
This paper 
shows that in dimensions $n \ge 2$
for any partition of the set of points in the standard $n$-sphere
$\sum_{i=0}^n x_i^2 =1$ in $\R^{n+1}$  into $(n+3)$ or more nonempty sets, there exists a 
hyperplane in $\R^{n+1}$ that intersects at least $(n+2)$ of these sets. 
This result is used to prove a result in inversive geometry.  A mapping $T: \mathbb{S}^2 \to \mathbb{S}^n$, for $n\geq 2$,
not assumed continuous or even measurable,  is called weakly circle-preserving if the image of any 
circle is contained in some circle in the range space $\mathbb{S}^n$. If such a map $T$ has a range $T(\mathbb{S}^2)$ in circular general position, meaning that any circle in the $\mathbb{S}^n$ misses
at least two points of $T(\mathbb{S}^2)$, then $T$ must be a \m transformation of $\mathbb{S}^2$.
\end{abstract}

\maketitle
\section{Introduction}

For definiteness
the  standard $n$-sphere ${\mathbb S}^n$ embedded in $\mathbb{R}^{n+1}$ is the real algebraic set 
$$
\mathbb{S}^n := \{ (x_0, x_1,  \cdots , x_{n})\in \R^{n+1} : \sum_{i=0}^{n} x_i^2 =1\}.
$$

The main result of this paper concerns properties of colorings of points of the $n$-sphere,
and can be stated in two ways, a coloring form and a sphere cutting form.


We start with the coloring form.
Conventionally a coloring of a manifold is a partition into well-defined  sets
which are the color classes.   The cardinality of the coloring is the index set of the partition, and is often 
appended to the description of the coloring, 
We define a   {\em full  $k$-coloring} of the $n$-sphere 
to be  a partition of $\mathbb{S}^n$  into $k$ nonempty sets: every color must be used.
That is, it is given by a map $\Gamma: X \to \{ 1, 2, ..., k\}$ which is onto the given range.
Here we allow arbitrary partitions in which 
the sets are not restricted in any way.

%
%

\begin{theorem}\label {th11b} 
{\rm (Main Theorem: Coloring Form)}
Let $n \ge 2$ and the points of a standard $n$-sphere $\mathbb{S}^n$  be colored
with a full $k$-coloring. If  $k\geq n+3$ then there exists some  $(n-1)$-sphere 
that contains at least $(n+2)$ different  colors.
\end{theorem}

Here an $(n-1)$-sphere in $\mathbb{S}^n$ is an element of the orbit
of the $(n-1)$-sphere $x_0=0$ in $\mathbb{S}^n$ under the action of the generalized M\"{o}bius group
acting on ${\mathbb S}^n$.
Alternatively it is  the intersection of a 
sphere with some hyperplane in $\R^{n+1}$ that contains an open set inside 
the sphere $\mathbb{S}^n$ in $\mathbb{R}^{n+1}$.
Notice this hyperplane need  {\em not } pass through the origin.



We can alternatively  re-state the main result in terms of partition of the points of the $n$-sphere, 
viewed as embedded in ${\mathbb R}^{n+1}$ as follows.

%
\begin{theorem}\label{th11} 
{\rm (Main Theorem: Sphere-Cutting Form)}
Let  $n\geq 2$ and 
let the standard $n$-sphere 
$\mathbb{S}^n$ be partitioned as a set  into $k$ nonempty subsets.
If $k \ge n+3$, then  there exists a  hyperplane
in ${\mathbb R}^{n+1}$ that intersects at least $(n+2)$ of these sets.
\end{theorem}

Here we note  that given any $(n+1)$ points on the $n$-sphere $\mathbb{S}^n$ one can 
always find a hyperplane 
in $\mathbb{R}^{n+1}$ passing through them, hence one can trivially guarantee 
 to intersect $(n+1)$ or more of  the sets in the partition. The content of the theorem lies 
in replacing $(n+1)$ by $(n+2)$ in the conclusion. 

In the paper  we shall prove the main theorem in the coloring  form.
It turns out  the  theorem becomes harder to prove as the dimension gets smaller. In Section \ref{sec2}  we give a proof
valid for dimensions $n \ge 3$. In Section \ref{sec3} we give a proof for the most interesting and hardest  case $n=2$.
The result does not hold in dimension $n=1$.

In Section \ref{sec6} we 
present examples  showing that the theorem is sharp in two senses.
First, if the theorem's hypothesis is weakened from  $(n+3)$ to $(n+2)$, then the
conclusion of the theorem no longer holds.
Second,  no matter how large a lower bound is placed on the number of sets $k$ in the partition,
one can never improve the conclusion of the   theorem to replace  $(n+2)$ with $(n+3)$.

\subsection{ Inversive Geometry: Generalized M\"{o}bius Group}

The main  result has a connection with  inversive geometry. Inversive geometry is the study of
geometric properties that are preserved under the action of the generalized \m group which will be defined in the next paragraph.
This is the group of all transformations of space $\mathbb{S}^n$ that map (generalized) $(n-1)$-spheres
to (generalized) $(n-1)$-spheres. 
The connection is that  the set of all $(n-1)$-spheres of a standard $n$-sphere form a single
orbit under the action of the group of inversions. In this framework Theorem \ref{th11b}
asserts that for a full  $(n+3)$-coloring,  the orbit of a single $(n-1)$-sphere 
under the generalized M\"{o}bius group contains some element requiring $(n+2)$-colors.

Under stereographic projection, we can identify $\mathbb{S}^n$ with $\R^n_\infty:=\R^n\cup\{\infty\}$. 
In this space, an \begin{em} inversion \end{em}(or a \begin{em}reflection\end{em}) in an $(n-1)$-sphere $S(a,r):=\{x\in\R^n:\|x-a\| = r\}$ is the function $\phi$ defined by $\phi(x) = a+(\frac{r}{\|x-a\|})^2(x-a)$. $\phi$ is well defined on $\R^n_\infty - \{a,\infty\}$, and at these two points, we define $\phi(a) = \infty$ and $\phi(\infty) = a$. A \begin{em}reflection\end{em} in a hyperplane 
is the usual reflection defined in $\R^n$ and fixes the point $\infty$.

We define  a {\em \m transformation} on $\R^n_\infty \cong \mathbb{S}^n$ to be a finite composition of reflections in 
$(n-1)$-spheres or hyperplanes. The group of all \m transformations is called 
the Generalized \m Group $GM(\R^n_\infty)$, following Beardon \cite[Chapter 3]{B83}.

In dimension two, with identification of $\R^2$ with $\mathbb{C}$, \m transformations include all
 the linear fractional transformations $z\mapsto \frac{az+b}{cz+d}$,  which are orientation-preserving maps 
 and also the conjugate ones $z\mapsto\frac{a\bar z+b}{c\bar z+d}$,
  which are orientation reversing maps. The former are holomorphic maps and the latter anti-holomorphic.
  We should be aware of the subtle difference between \m transformations and linear fractional 
  transformations (which sometimes, unfortunately, are also called \m transformation in complex analysis) in dimension two.

The generalized \m group is closely related to the conformal group on $\mathbb{S}^n$.
 Any \m transformation is a conformal diffeomorphism on $\mathbb{S}^n$. A result of  Liouville\cite{Lio1850}  in 1850 
 asserts (in modern form) a  local converse: when $n\geq 3$, any smooth conformal diffeomorphism of a simply
 connected open domain $U$ of ${\mathbb S}^n$ into ${\mathbb S}^n$ is 
 the restriction of a \m transformation. In particular, for $n\geq 3$, the conformal group on $\mathbb{S}^n$ 
 is precisely the generalized \m group.

\subsection{Analogue of Main Theorem in Euclidean Geometry}\label{main-thm-euclid}

The lower bound of $n+2$ colors found in the conclusion of Theorem \ref{th11b}   is a result in inversive geometry insofar
as it is associated to 
the generalized \m group. For Euclidean geometry in $\R^{n+1}$,
the set of isometries preserving the standard $n$-sphere is given by the action of the orthogonal group. 
Under this action, all $(n-1)$-spheres in $\mathbb{S}^n$ no longer form a single orbit. Nevertheless, 
the orthogonal group acts transitively on {\em great $(n-1)$-spheres}, which consist of 
intersections of the standard $n$-sphere with hyperplanes passing
through the origin. Among the orbits of great $(n-1)$-spheres, there is a  parallel result: 
\begin{theorem}\label{great-sphere-thm}
For any full $k$-coloring of the standard $n$-sphere with $k\geq n+2$, there exists some great $(n-1)$-sphere that contains at least $(n+1)$ different colors. 
\end{theorem}

 This result is also completely sharp in two senses.
First, if the hypothesis of Theorem \ref{great-sphere-thm}  is weakened from  $(n+2)$ to $(n+1)$, then the
conclusion of the theorem no longer holds. 
Second,  no matter how large a lower bound is placed on the number of sets in the partition,
one can never improve the conclusion of the   theorem to replace  $(n+1)$ with $(n+2)$. 
The proof of Theorem \ref{great-sphere-thm}  is quite simple and we present it in
 section \ref{euclid}.

We note  that the group $O(n)$  is a Lie group of lower rank than that for inversive geometry.
This  may supply one reason why  the proof of Theorem \ref{th11b} 
in the two-dimensional case seems to necessarily be complicated,
and harder than the Euclidean geometry problems  listed above.


\subsection{Application: Characterizing M\"{o}bius transformations on the $2$-sphere}
We give an application of the main theorem for $n=2$ which motivated this work.
It   gives a new  characterization of \m transformations in the plane. 
Note that our definition of \m transformation (when identifying $\R^2$ with $\mathbb{C}$)
a single reflection $z\mapsto \bar{z}$  is a \m transformation,
although it is anti-holomorphic. 
There has been much literature on the characterization of such transformations, for example the
1937 result of Caratheodory \cite{Car37}.  We may also mention
related work done  in recent years.
Many papers in the literature  consider {\em circle-preserving maps},
 which require that the image of a circle to be another circle or a point,  for example, the results in
 \cite{BM01}, \cite{CP99}, \cite{LY09}, \cite{Yao07}, \cite{Yao11}. 

This characterization we consider  is in terms of   weakly circle-preserving maps, which we now define.

\begin{defn} 
A function $T:\mathbb{S}^2\longrightarrow \mathbb{S}^n$ with $n\geq 2$ is called {\em weakly circle-preserving} if for every circle $C \subset \mathbb{S}^2$, $T(C)$ lies in some circle $C'$ in $\mathbb{S}^n$.
\end{defn}
In particular, any  function whose image is finite and
 contains   $3$  or fewer points is weakly circle-preserving. 
We should notice that in this definition we do not assume  that $T$ is injective or
 continuous, or even Borel-measurable. This notion is a much more relaxed notion than has been considered elsewhere,
 and it includes many maps that are not \m transformations. 
  It was originally introduced by the first author with Webb in \cite{GW79} (also \cite{Gib82}), using however
 the name ``circle-preserving map".

We will show that a
suitable  ``general position"
condition on the range $T(\mathbb{S}^2)$ of a weakly circle-preserving  map  is sufficient to force
it to be a \m transformation. 

\begin{defn}
A subset $B$ of the $n$-sphere is said to lie in {\em circular general position} if for each  circle  $C$, the complement of $C$ contains at least two points of $B$.
\end{defn}

The definition implies that any $B$ in circular general position must contain at least $5$ points.

In 1979 the first author with Webb \cite{GW79} proved the following  ``six-point theorem".
\begin{theorem}\label{6point}
   {\em (``Six-point theorem")}
Let  $T$ be a weakly circle-preserving map from  ${\mathbb S}^2$ into  ${\mathbb S}^n$ with $n\geq 2$
which satisfies the following two conditions. 
\begin{enumerate}
\item
Every circle in the codomain ${\mathbb S}^n$ does not contain
at least two points in the image $T({\mathbb S}^2)$, i.e.  
the image $T({\mathbb S}^2)$ is in circular general position in $\mathbb{S}^n$. 

\item
The image $T({\mathbb S}^2 )$ contains at least six distinct points. 
\end{enumerate}
Then $T({\mathbb S}^2 )$ is a $2$-sphere and $T$ is a \m transformation.
\end{theorem}

\begin{proof} 
This is a special case  of Theorem 1 of \cite{GW79}. That theorem allowed an  
open domain $U= {\mathbb S}^2$ and allowed a range inside ${\mathbb S}^n$ for any $n \ge 2$.
Here we take $U= \mathbb{S}^2$ and $n=2$.  We change notation $f$ to $T$
and change  ``inversive transformation"  to  \m
transformation.
\end{proof}

The main result of this paper allows us to improve the result above, as follows. 
\begin{theorem}\label{5point}
 {\em (``Five-point theorem")}
Let  $T$ be a weakly circle-preserving map from  ${\mathbb S}^2$ into  ${\mathbb S}^n$ with $n\geq 2$
which satisfies the following conditions. 
\begin{enumerate}
\item
Every circle in the codomain ${\mathbb S}^n$ does not contain
at least two points in the image  $T({\mathbb S}^2)$, i.e.
$T({\mathbb S}^2)$ is in circular general position in $\mathbb{S}^2$. 
\item
The image $T({\mathbb S}^2 )$ contains five or more   distinct points. 
\end{enumerate}
Then $T({\mathbb S}^2 )$ is a $2$-sphere and $T$ is a \m transformation.
\end{theorem}

We have stated Theorem \ref{5point}  in parallel with Theorem \ref{6point}, but condition (2) is
not needed, as it is implied by condition (1). 


Theorem \ref{5point} looks like only a small improvement over the six point theorem. 
However this improvement is important for two reasons.  First,  this  result is sharp:
 we cannot further improve $5$ to $4$, as shown
 by Proposition \ref{sharp-map} in the section~\ref{sec4}.
Second,  this result permits proving a very strong result in inversive geometry
characterizing \m transformations in higher dimensions, subsuming 
previously known results, which we will present
in \cite{GL13b}. This result uses an inductive argument on dimension which requires
the result above for the base case.

\subsection{Discussion}

The statements of the main results above have a flavor of Ramsey theory or anti-Ramsey theory
but do not fall under any of the topics considered up to now in this subject.

Ramsey theory concerns the study of finite partitions of sets
and of  finding monochromatic objects in a single partition class.
It can be combined with geometry. 
The subject of {\em Euclidean Ramsey theory}, initiated by Erd\H{o}s et al. \cite{EGMRSS}  addresses the following problem. 
Consider a four-tuple $(H,X,m,n)$, where $H$ is a subgroup of the Euclidean group of $\mathbb{R}^n$, 
$X$ is a finite subset of $\mathbb{R}^n$, and $m$ is the cardinality of a coloring of $\mathbb{R}^n$. The four-tuple
$(H, X, m, n)$  is said to have the {\em Ramsey property} if the orbit of $X$ under $H$ necessarily contains a 
monochromatic set, for every full coloring of cardinality 
$m$. One problem in Euclidean Ramsey Theory consists of discovering and characterizing such four-tuples.
One can now propose analogous problems for {\em Inversive Ramsey Theory},
which replaces the Euclidean group with  the generalized \m group. 

The subject of {\em  anti-Ramsey theory} concerns finding polychromatic objects: finite sets of points
such that there exists some copy under isometry group which all elements are in different color classes
for allowable partitions (i.e. different elements of the partition). 
This subject has been studied for finite
graphs, where partitions of vertices are also finite sets,  see
 Erd\H{o}s, Siminovits and S\'{o}s \cite{ESS73} and Babai \cite{Bab85}.
  In the  geometric situation that we consider, with infinite sets, and where nevertheless some sets in the partition may have only one element, it seems impossible to guarantee
all elements of a set to fall in different color classes. However the main results above do
have an anti-Ramsey like flavor.


\subsection{Notation} 

Throughout this paper, we will identify $\mathbb{S}^n$ with $\mathbb{R}^n_\infty:=\mathbb{R}^n\cup\{\infty\}$, 
and a $k$-sphere in $\mathbb{S}^n$ with either Euclidean 
${\mathbb S}^k \subset \mathbb{R}^n$ or a $k$ dimensional affine space in $\mathbb{R}^n$ together 
with the point $\infty$. Once the identification is specified, we will use the term extended line for Euclidean line with $\infty$
 and extended $k$-plane for $k$-dimensional affine space in $\mathbb{R}^n$ with $\infty$.
 
A full $k$-coloring of $\mathbb{S}^n$ is a surjective map $\Gamma:\mathbb{S}^n \longrightarrow \{1,...,k\}$. We will say a point $x$ is of color $i$ if $\Gamma(x) = i$, a set $X\subset \mathbb{S}^n$ contains color $i_1,...,i_m$ if $\{i_1,...,i_m\}\subset \Gamma(X)$, and a set $X\subset \mathbb{S}^n$ is colored by $i_1,...,i_m$ if $\Gamma(X) = \{i_1,...,i_m\}$. We will also say a set $X$ is of $m$ colors if $|\Gamma(X)| = m$, and $X$ contains $m$ colors if $|\Gamma(X)|\geq m$. We will sometimes use $\Gamma_i$ to simplify the notation $\Gamma^{-1}(i)$. We will usually use the subscript to describe the color of a point (e.g., $x_i$ is a point of color $i$).
 
 Given a finite set of points $\{x_1,...,x_n\}$ in $\R^n_\infty (= \mathbb{S}^n)$, we will use $(x_1x_2...x_n)$ to denote the smallest dimension (generalized) sphere containing $\{x_1,...,x_n\}$. Notice that $\dim((x_1x_2...x_n))$ is not necessarily $n-2$, however, we do know the dimension is less than or equal to $n-2$. We will simply use $[x_1x_2...x_n]$ for $(x_1x_2...x_n\infty)-\{\infty\}$, i.e., the smallest dimension affine space containing $\{x_1,...,x_n\}$.
 
Sometimes, for convenience, we will also denote a sphere by its colors (e.g., $(123)$ is used to denote a circle of color 1,2 and 3). Similarly, we will also denote an affine space by its colors. We will make it clear whenever we use this notation to make sure the notation is assigned to a specific sphere of these colors.

\section{Proof of Theorem \ref{th11b} for $n\geq 3$}\label{sec2}

The proof of the theorem for $n\geq 3$ is easier than the theorem for $n=2$. The argument is purely combinatorial.
We begin with  the following definition and lemma.

\begin{defn}\label{def21}
Let $S\subset \mathbb{S}^n$ be a $(n-1)$-sphere, we say two points $x,y\in \mathbb{S}^n$ are {\em separated by $S$}
 if $x,y$ lives in two  different connected components of $\mathbb{S}^n-S$.
\end{defn}

\begin{lem}\label{separate-lem}
Let $\mathbb{S}^n$ be full $(n+3)$-colored ($n\geq 2$), assume there is no $(n-1)$-sphere containing $(n+2)$ colors, then there is an $(n-1)$-sphere $S$ of $(n+1)$ colors and two points $x_1, x_2$ of the rest two colors respectively such that $x_1,x_2$ are separated by $S$.
\end{lem}
\begin{proof}

We will consider 3 cases: $n=2$, $n=3$, $n>3$ separately. The case $n=3$ and the case $n>3$ can be treated together, however we feel that it is better to discuss the case $n=3$ first as it gives geometric meaning of the methods in higher dimensions. \\

Case $n=2$:

Let $x_1,x_4,x_5$ be three points colored by color 1,4 and 5. Under a \m transformation, we may assume that $x_1,x_4,x_5$ lie on a line and $x_1$ lies between $x_4$ and $x_5$ (recall here that we have identified $\mathbb{S}^2$ with $\R^2_\infty$). Denote this line by $[x_1x_4x_5]$.

Let $x_2,x_3$ be two points colored by 2 and 3, then $x_2,x_3\notin [x_1x_4x_5]\cup \{\infty\}$, as otherwise the circle $(x_1x_4x_5) = [x_1x_4x_5]\cup\{\infty\}$ contains 4 distinct colors.

If $x_2$ and $x_3$ lie on the opposite side of the line $[x_1x_4x_5]$ in $\R^2$, then $x_2$ and $x_3$ are separated by the circle $[x_1x_4x_5]\cup\{\infty\}$ in $ \R^2_\infty (=\mathbb{S}^2)$.

Otherwise, $x_2$ and $x_3$ lie on the same side of the line $(x_1x_4x_5)$ in $\R^2$.

Notice that $\angle x_4x_2x_5\neq \angle x_4x_3x_5$, as otherwise, the four points $x_2,x_3,x_4,x_5$ will be on a same circle. Without loss of generality, assume $\angle x_4x_2x_5< \angle x_4x_3x_5$, then the point $x_2$ is outside the circle $(x_3x_4x_5)$. Since $x_1$ is between $x_4$ and $x_5$, $x_1$ is inside the circle $(x_3x_4x_5)$. Hence, $x_1$ and $x_2$ are separated by the circle $(x_3x_4x_5)$ (see Figure~\ref{fig1}).
\begin{figure}[h!]
\begin{center}
\includegraphics[width=70mm]{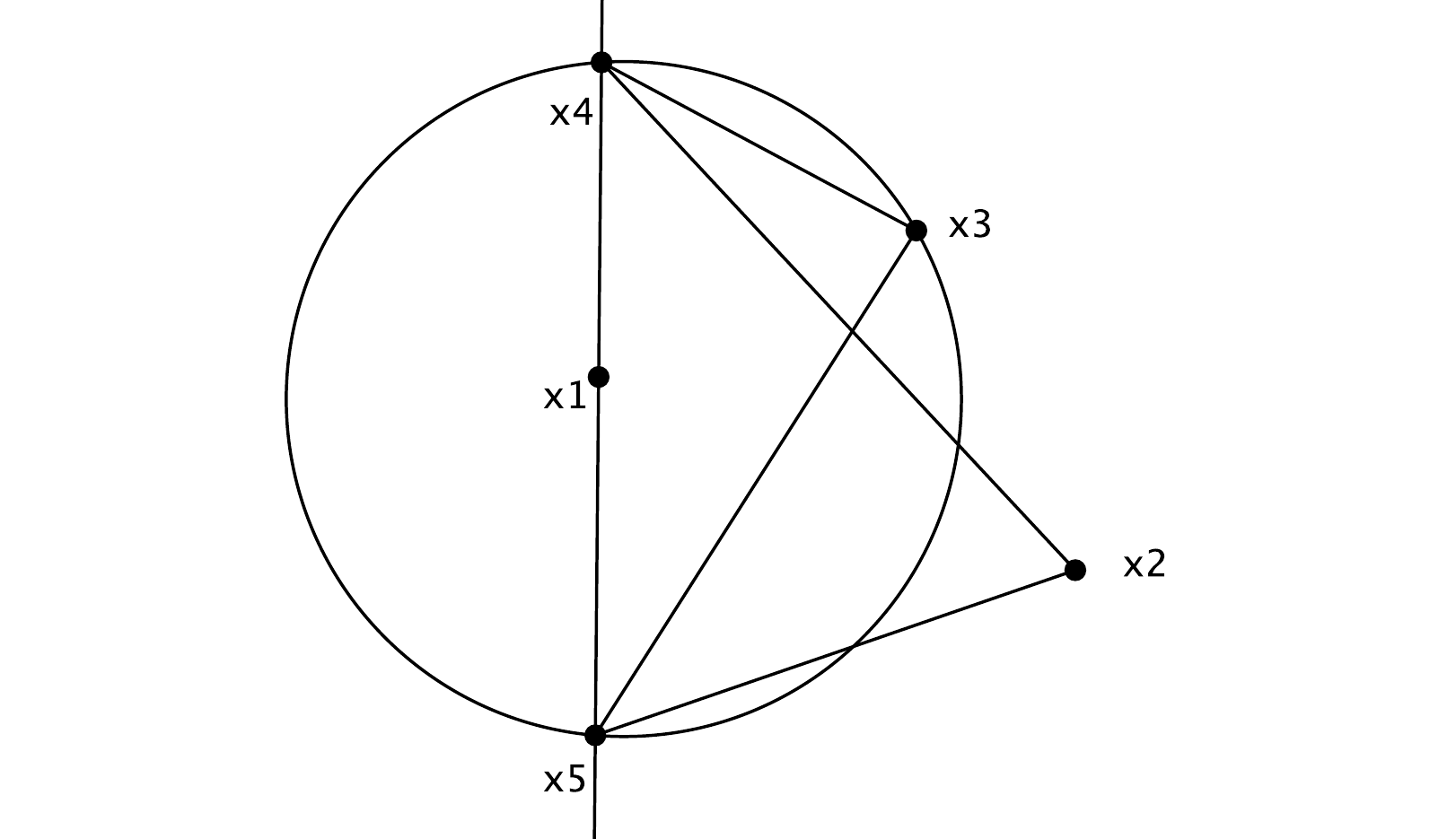}
\end{center}
\caption{Two Points Separated by a Circle}
\label{fig1}
\end{figure}

Before proving the case $n=3$ and $n>3$, we shall prove the following observation.\\

\begin{bf}
Observation:
\end{bf}
If $\mathbb{S}^n$ is full $(n+3)$-colored with no $(n-1)$-sphere containing $(n+2)$ or more colors, then there is no $k$-sphere containing $(k+3)$ colors for all $k = 2,3,...,n-2$.

\begin{proof} (of the observation:)

To see this, we notice that suppose for contradiction that we have a $(k+3)$-colored $k$-sphere $S_k$, then we can form $(k+1)$-sphere using this $S_k$ and a point with different color to get a $(k+4)$-colored $(k+1)$-sphere. So inductively, we would get a $(n+2)$-colored $(n-1)$-sphere which is a contradiction.

\end{proof}

Case $n=3$:

Let $x_1,x_5,x_6 \in \R^3_\infty$ colored by 1, 5 and 6. Under a \m transformation, we may assume that $x_1,x_5,x_6$ lie on a line and $x_1$ lies between $x_5$ and $x_6$. Denote this line, as usual, by $[x_1x_5x_6]$.

Let $x_2,x_3,x_4 \in \R^3_\infty$ colored by 2, 3 and 4. Notice that $x_2,x_3,x_4\notin [x_1x_5x_6] \cup\{\infty\}$, so $[x_1x_4x_5x_6]$ is a 2-plane.

If $x_2$ and $x_3$ lie on the opposite side of the 2-plane $[x_1x_4x_5x_6]$ in $\R^3$, then $x_2$ and $x_3$ are separated by the $2$-sphere $[x_1x_4x_5x_6]\cup\{\infty\}$.

Otherwise, $x_2$ and $x_3$ lie on the same side of the plane $[x_1x_4x_5x_6]$ in $\R^3$

Notice that $x_4, x_5, x_6$ do not lie in a line, so $(x_4x_5x_6)$ is a Euclidean circle in $\R^2$. Let $O$ be the center of this circle. We let $l$ be the line perpendicular to the plane $[x_1x_4x_5x_6]$ through the point $O$. Let $A, B$ be two points on the circle $(x_4x_5x_6)$ such that $O$ is on $AB$. We rotate $x_2$ and $x_3$ about the line $l$ to $x_2'$ and $x_3'$ so that $x_2'$ and $x_3'$ lie on the plane spanned by $AB$ and $l$. Then any sphere containing the circle $(x_4x_5x_6)$ has its center on $l$, so the sphere is invariant under rotation about $l$. Hence, the point $x_2$ or $x_3$ is inside, on or outside a sphere containing $(x_4x_5x_6)$ if and only if $x_2'$ or $x_3'$ is.

Now working on the plane $[l,A,B]$, notice $x_2'$ and $x_3'$ are on the same side of the line $AB$. Also notice that $\angle Ax_2'B \neq \angle Ax_3'B$ as otherwise, $x_2, x_3, x_4, x_5, x_6$ will be on a $2$-sphere. Without loss of generality, assume $\angle Ax_2'B < \angle Ax_3'B$, then $x_2'$ is outside the circle $(x_3'AB)$. Notice that $(ABx_3')$ is a great circle of the sphere $(x_3'x_4x_5x_6) = (x_3x_4x_5x_6)$, so $x_2'$ is outside the sphere $(x_3x_4x_5x_6)$ (see Figure~\ref{fig2}). Hence $x_2$ is outside the sphere $(x_3x_4x_5x_6)$. Notice that $x_1$ is between $x_5$ and $x_6$, so $x_1$ is inside the sphere $(x_3x_4x_5x_6)$. Therefore, $x_1$ and $x_2$ are separated by $(x_3x_4x_5x_6)$.\\
\begin{figure}[h!]
\begin{center}
\includegraphics[width=70mm]{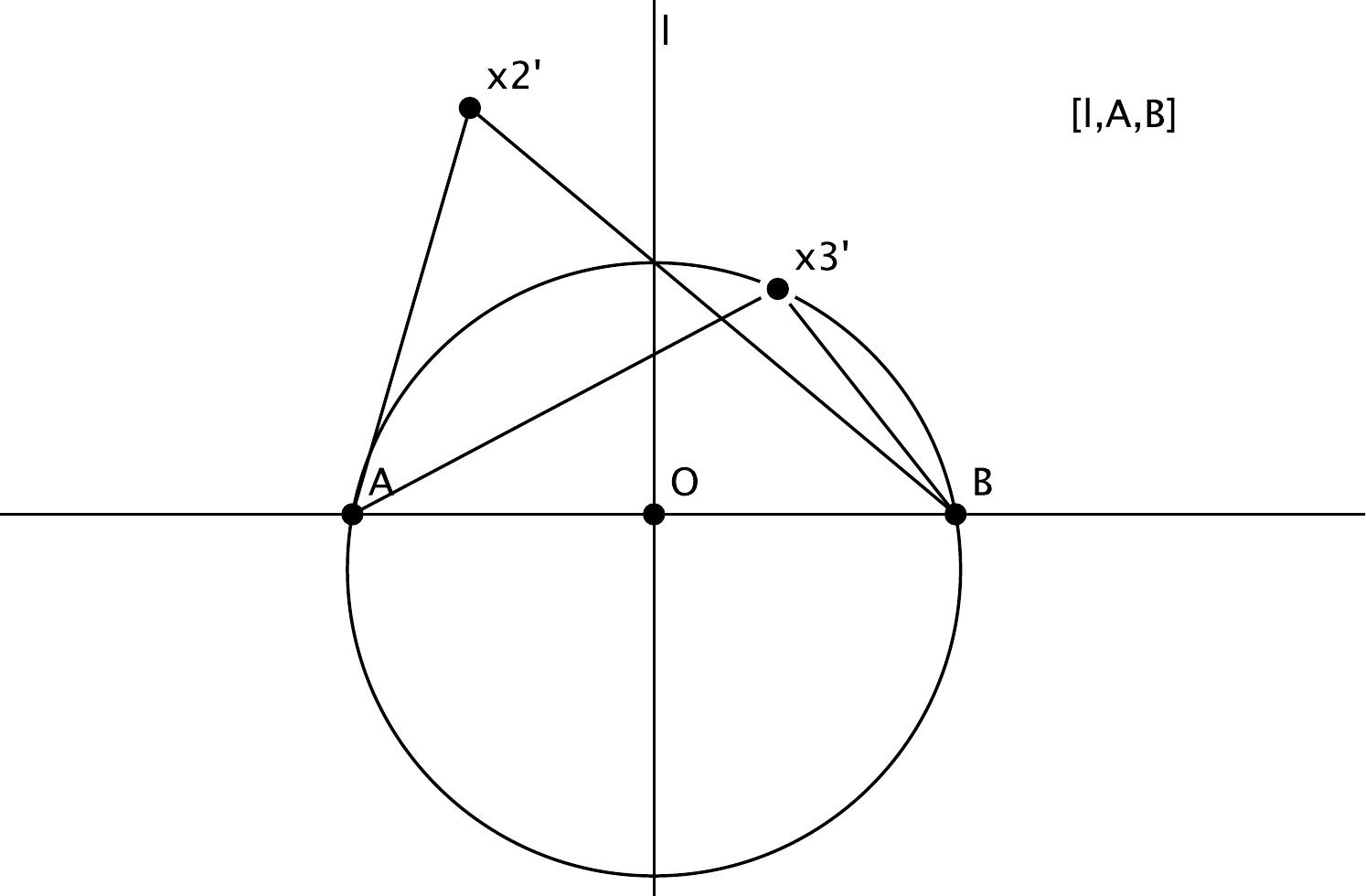}
\end{center}
\caption{Cross Section of the Sphere}
\label{fig2}
\end{figure}

Case $n>3$:

The proof of this case is essentially the same as the case $n=3$. We first find three points $x_1,x_{n+2},x_{n+3}$ of color $1, n+2$ and $n+3$. Under a \m transformation, we may assume that $x_1,x_{n+2},x_{n+3}$ lie on a line and $x_1$ lies between $x_{n+2}$ and $x_{n+3}$. Denote this line, as usual, by $[x_1x_{n+2}x_{n+3}]$.

Let $x_2,...,x_{n+1}$ be points of color $2,...,n+1$ respectively. Notice that $x_2,...,x_{n+1}\notin [x_1x_5x_6] \cup\{\infty\}$. Now consider the affine space $[x_1x_4x_5...x_{n+3}]$. $[x_1x_4x_5...x_{n+3}]$ can have dimension at most $n-1$ as $x_1, x_{n+2}$ and $x_{n+3}$ are on a line. $[x_1x_4x_5...x_{n+3}]$ also cannot have dimension $n-2$ or less as otherwise, $[x_1x_4x_5...x_{n+3}]\cup\{\infty\}$ is a sphere of dimension $n-2$ or less with at least $n+1$ colors. Therefore, $[x_1x_4x_5...x_{n+3}]$ is a hyperplane.

If $x_2$ and $x_3$ lie on the opposite side of the hyperplane $[x_1x_4x_5...x_{n+3}]$ in $\R^n$, then $x_2$ and $x_3$ are separated by $[x_1x_4x_5...x_{n+3}]\cup\{\infty\}$.

Otherwise, $x_2$ and $x_3$ lie on the same side of the hyperplane $[x_1x_4x_5...x_{n+3}]$ in $\R^n$

Notice that $[x_4x_5...x_{n+3}]$ is again a hyperplane as $x_1, x_{n+2}, x_{n+3}$ are on a line. Therefore, $(x_4x_5...x_{n+3})$ is either an extended $(n-1)$-plane or a Euclidean $(n-2)$-sphere. But any $n$ points can determine a sphere of dimension at most $(n-2)$, so $(x_4x_5...x_{n+3})$ is a Euclidean $(n-2)$-sphere. Let $O$ be the center of $(x_4x_5...x_{n+3})$, $l$ be the line perpendicular to the hyperplane $[x_1x_4x_5...x_{n+3}]$ through $O$ and $A,B$ be two points on $(x_4x_5...x_{n+3})$ such that $O$ is on $AB$.

In dimension $n\geq 4$, we define a rotation about a line to be the rotation in $\R^n$ that preserves all hyperplanes perpendicular to the line, and we should call this line as the axis of rotation. E.g., a rotation about the $X_1$-axis is a matrix of the form
\[ \left( \begin{array}{cc}
1 & 0 \\
0 & SO_{n-1}(\R) \end{array} \right)\]

Notice that the orbit of a point under all the rotation about a line is a $(n-2)$-sphere contained in some hyperplane perpendicular to the axis of rotation. 

Let $P$ be the hyperplane through $x_2$ that perpendicular to $l$, then $[l,A,B]\cap P$ is a line passing through the center of the orbit of $x_2$ under rotation about $l$. Hence, $[l,A,B]$ intersects the orbit of $x_2$. Similar result holds for $x_3$. Therefore, we may rotate $x_2$ and $x_3$ about $l$ to get $x_2'$ and $x_3'$ which are on the plane $[l,A,B]$. Similar to the case $n=3$,  the center of any $(n-1)$-sphere containing $(x_4x_5...x_{n+3})$ lies on $l$, and the the orbit of $x_2$ (or $x_3$) is an $(n-2)$-spheres perpendicular to $l$ with center on $l$, so $x_2$ (or $x_3$) is inside, on or outside $(n-1)$-sphere containing $(x_4x_5...x_{n+3})$ if and only if $x_2'$ (or $x_3'$) is.

Now working on the plane $[l,A,B]$, following exactly the same argument in the case $n=3$, if we choose the point with larger angle (say $x_3'$), we have that $x_2$ is outside the $(n-1)$-sphere $(x_3x_4...x_{n+3})$. But $x_1$ is between $x_{n+2}$ and $x_{n+3}$, so it is inside the $(n-1)$-sphere $(x_3x_4...x_{n+3})$. Hence, $x_1$ and $x_2$ are separated by $(x_3x_4...x_{n+3})$.
\end{proof}

Now it is easy to prove Theorem \ref{th11b} for $n \ge 3$.

\begin{proof} (of Theorem \ref{th11b} in dimension $\geq 3$)

It is easy to see that we only need to prove when $k = n+3$, the case $k>n+3$ follows immediately from the case $k = n+3$.

We suppose for contradiction that there is no $(n-1)$-sphere containing $n+2$ colors.

We will first consider the case when $n=3$

By lemma \ref{separate-lem}, there is a 4-colored $2$-sphere $S$ (say it is colored by 3,4,5,6) and two points (say $x_1, x_2$ of color 1,2) such that $x_1$ and $x_2$ are separated by $S$. Now the line $[x_1x_2]$ necessarily intersects the sphere at two points of the same color, as otherwise, $[x_1x_2]$ will contain 4 colors. Without loss of generality, we assume that two intersection points are of color 3. Under a \m transformation, we may map one of the intersections to infinity and the other to $\vec 0$ and $x_1$ to $(1,0,0)$. Under this transformation, the $2$-sphere $S$ is mapped to an extended 2-plane through $\vec 0$, and $\vec 0$ is between $x_1$ and $x_2$.

Let $x_4,x_5,x_6$ be on this extended plane $S$ having color 4,5 and 6. We claim the line $[\vec 0x_6]$ must be bicolored, i.e., $[\vec0 x_6]$ is colored by color 3 and 6 only. Suppose not, if either color 1 or color 2 is on $[\vec 0x_6]$, then the extended 2-plane $S$ will have 5 colors. Otherwise, either color 4 or color 5 is on $[0x_6]$, then extended plane $(\vec 0x_1x_2x_6) = [\vec 0x_1x_2x_6]\cup\{\infty\}$ will have 5 colors. This is a contradiction. Hence $[\vec 0x_6]$ is bicolored.

Now if $x_1,x_2,x_4,x_5$ lie on a same plane, then we will have the extended plane $ [x_1x_2x_4x_5]\cup\{\infty\}$, which has dimension $\leq 2$, contains color 1,2,3,4,5.

Otherwise $(x_1x_2x_4x_5)$ is a Euclidean sphere, the line $[\vec 0x_6]$ intersects $(x_1x_2x_4x_5)$ at two points. These two points must be colored by 3 or 6. So $(x_1x_2x_4x_5)$ will have at least 5 colors. This is a contradiction, and the case for $n=3$ follows.\\

For the case $n>3$, the argument goes in a similar way. First by lemma \ref{separate-lem}, there is a $(n+1)$-colored $(n-1)$-sphere $S$ (say it is colored by $3,4,...,n+3$) and two points (say $x_1$ and $x_2$ of color 1 and 2) such that $x_1$ and $x_2$ are separated by $S$. The line $[x_1x_2]$ must intersect $S$ at two points of same color (say the color is 3). Under a \m transformation, we map one of the intersections to infinity and the other to $\vec 0$ and $x_1$ to $(1,0,...,0)$. Under this transformation, $S$ is mapped to an extended $(n-1)$-plane through $\vec 0$ and $\vec 0$ is between $x_1$ and $x_2$.

Let $x_4,...,x_{n+3}$ be on this extended plane $S$ having color $4,...,n+3$. We claim that $[\vec 0x_{n+3}]$ is bicolored. Suppose not, if either color 1 or 2 is on $[\vec 0x_{n+3}]$, then the extended $(n-1)$-plane $S$ will contain at least $n+2$ colors. Otherwise, one of color $3,...,n+2$ is on $[0x_{n+3}]$, then the extended 2-plane $(\vec 0x_1x_2x_{n+3}) = [\vec 0x_1x_2x_{n+3}]\cup\{\infty\}$ contains at least 5 colors. This is a contradiction to the observation. Hence $[\vec 0x_{n+3}]$ is bicolored.

Now if $x_1,x_2,x_4,x_5,..,x_{n+2}$ lie on a same $(n-1)$-plane, then the extended plane $[x_1x_2x_4x_5...x_{n+2}] \cup\{\infty\}$, which has dimension $\leq n-1$, contains color $1,2,...,n+2$.

Otherwise, $(x_1x_2x_4x_5...x_{n+2})$ is a Euclidean $(n-1)$-sphere, the line $[0x_{n+3}]$ intersects $(x_1x_2x_4x_5...x_{n+2})$ at two points. These two points must be colored by $3$ or $n+3$. So $(x_1x_2x_4x_5...x_{n+2})$ will have at least $n+2$ colors. This is a contradiction, and the the case for $n>3$ follows.\\
\end{proof}

\section{Proof of Theorem \ref{th11b} for $n=2$}\label{sec3}

The proof of the theorem for $n=2$ is different from the case of $n\geq 3$ and more complicated. For convenience
we restate Theorem \ref{th11b}  in this case. 


\begin{theorem}\label{th31} 
Any full $n$-coloring with $n \ge 5$ of the standard $2$-sphere ${\mathbb S}^2$ must have some circle containing at least $4$
different  colors.
\end{theorem}

\begin{proof}

Recall that we identify $\mathbb{S}^2$ with $\R^2_\infty$. We will argue by
contradiction, and suppose for contradiction that there is a $5$-coloring $\Gamma:\R^2_\infty \longrightarrow \{1,...,5\}$ with no $4$-colored circle, and denote $\Gamma_i :=\Gamma^{-1}(i)$.
We shall derive more and more constraints on this coloring, doing this in a series of claims,  until we obtain a contradiction.

\begin{figure}
\begin{center}
\begin{subfigure}[h!]{0.4\textwidth}
\centering
\includegraphics[width=50mm]{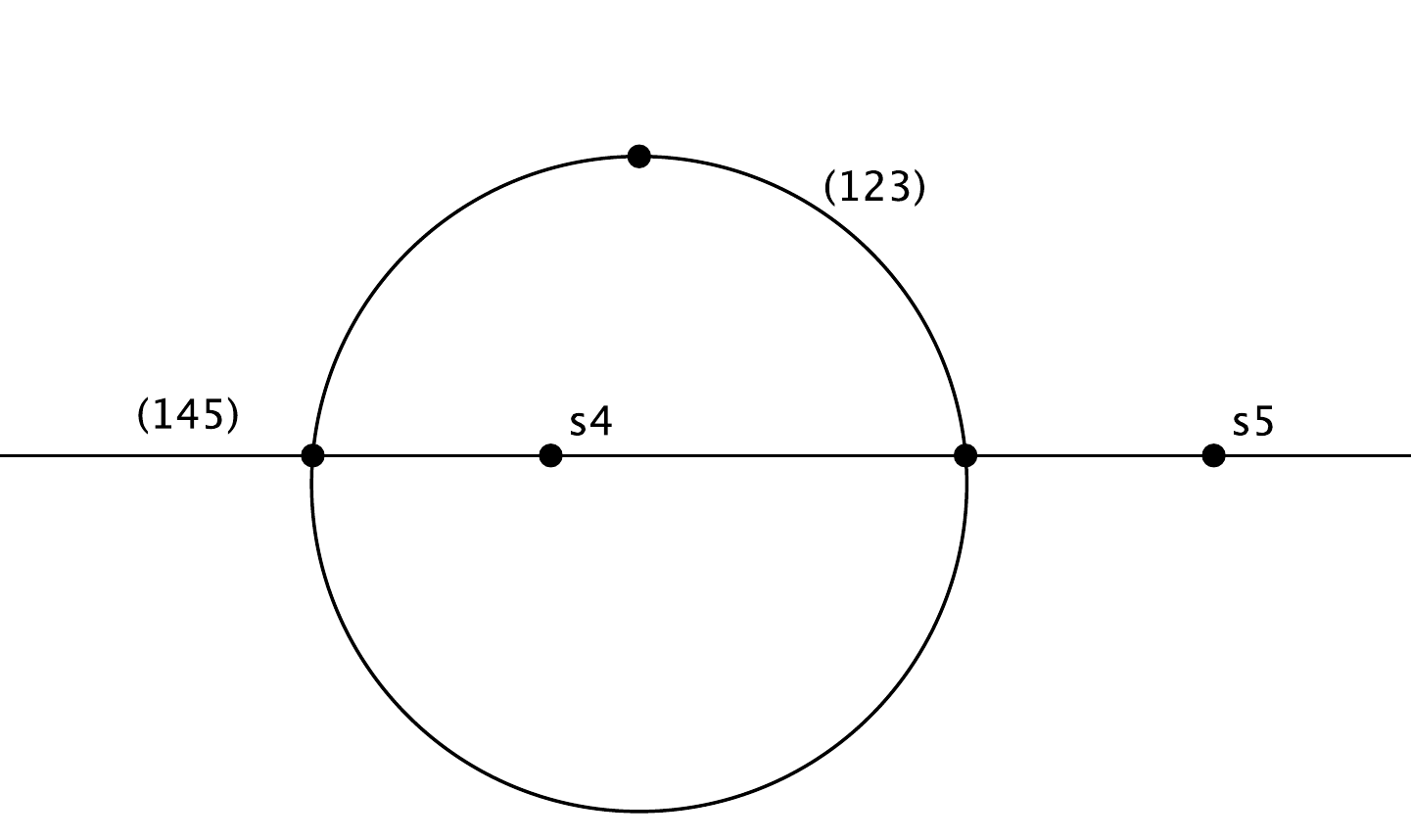}
\caption{Before Transformation}
\label{fig3a}
\end{subfigure}
\begin{subfigure}[h!]{0.4\textwidth}
\centering
\includegraphics[width=50mm]{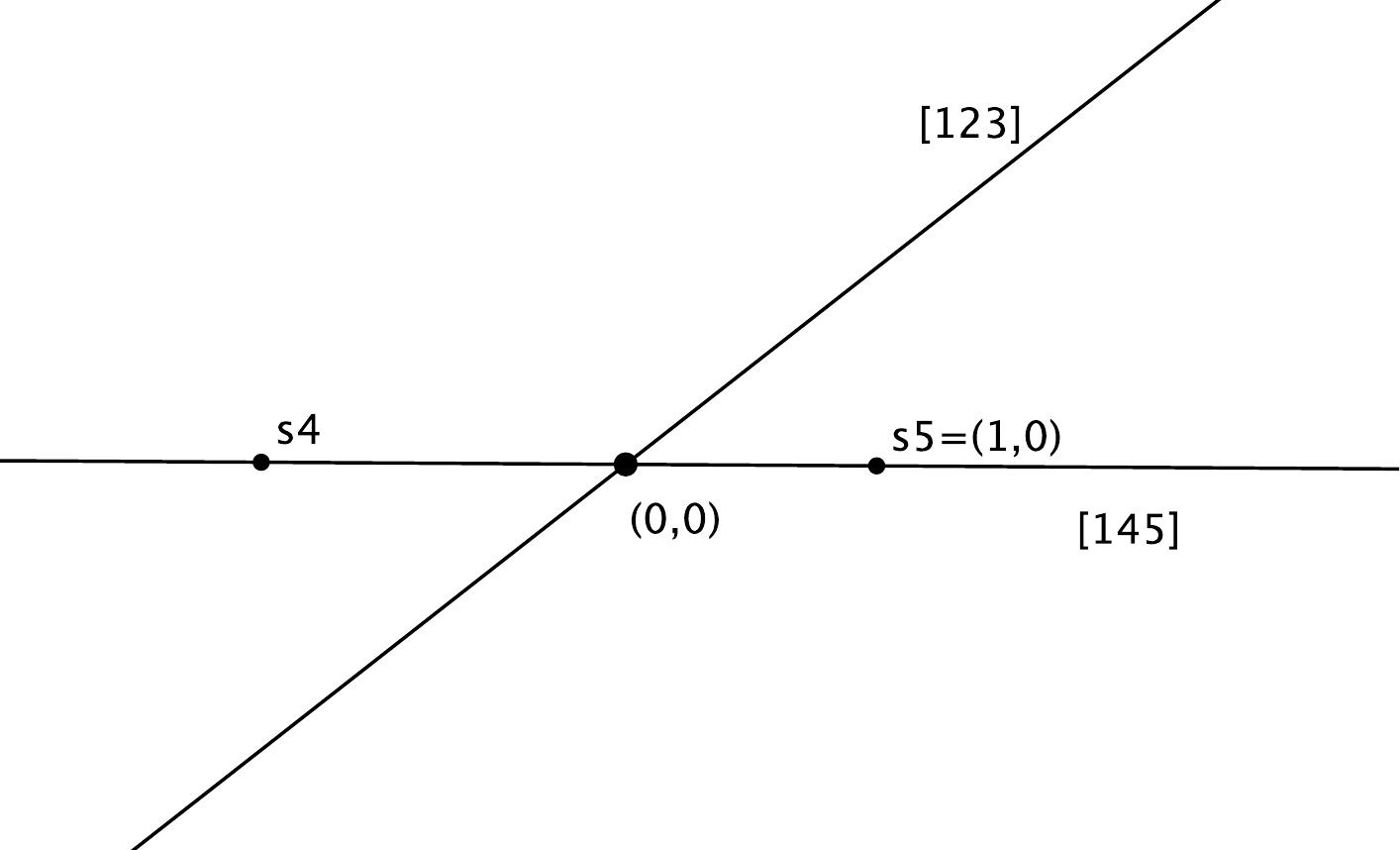}
\caption{After Transformation}
\label{fig3b}
\end{subfigure}
\end{center}
\caption{}
\label{fig3}
\end{figure}

To begin,  by Lemma \ref{separate-lem}, this $5$-coloring must contain  a 3-colored circle 
that separate two points of the other two colors. Without loss of generality, we assume that this circle is colored by 
colors 1,2 and 3, so we denote it by
$(123)$, and that there are 
two points $s_4$ and $s_5$  of colors $4$ and $5$, respectively, 
which are separated by the circle $(123)$. Consider the line 
$[s_4s_5]$, it necessarily intersects the circle $(123)$ at two points of the same color (say the color is $1$). 
We label the extended line $[s_4s_5]\cup\{\infty\}$ by its colors $(145)$.
Under a \m transformation, we map one of the intersections to $\vec 0$ and the other to $\infty$ and $s_5$ to $(1,0)$. Under this transformation, the circle $(123)$ is mapped to an extended line through origin,  the extended line $(145)$
is mapped  to the extended X-axis. We will also denote the Euclidean line after the transformation $[123] :=(123)-\{\infty\}$ and $[145]:=(145)-\{\infty\}$. (see Figure~\ref{fig3})

\begin{cl} \label{two-sided-lem}
$\Gamma_i$ intersects $[145]$ on both sides of $0$ and $\Gamma_j$ intersects $[123]$ on both sides of $0$ ($i=4,5$ and $j = 2,3$)
\end{cl}
\begin{proof} (of Claim \ref{two-sided-lem})

By the construction, there are points of color 5 on right hand side of $\vec 0$ (namely $s_5 = (1,0)$) and points of color 4 on left hand side of $\vec 0$ (namely $s_4$). We let $b_3$ be a point of color 3 on $[123]$, then the circle $(s_4s_5b_3)$ intersects the line $[123]$ at another point on the opposite side of $b_3$. The intersection point must be of color 3, so there are color $3$ on either side of $\vec 0$ of $[123]$. Similarly, the claim holds for color 2.

Now let $b_2$ and $b_3$ be on the opposite side of $\vec 0$ on $[123]$ of color 2 and 3, then the circle $(s_5b_2b_3)$ intersects the line $[145]$ at another point on the opposite side of $s_5$ and having color 5. Hence there are points of color 5 on either side of $\vec 0$ on $[145]$. Similarly, the claim holds for color 4.
\end{proof}

We need some extra notation to describe colored points on the two lines $[123]$ and $[145]$.
To distinguish the points on the line $[145]$ and $[123]$, we will denote an arbitrary point on the line $[145]$ by $a$ and points of color 4 and 5 on $[145]$ by $a_4$ and $a_5$. Similarly, we denote an arbitrary point on the line $[123]$ by $b$ and points of color 2 and 3 on $[123]$ by $b_2$ and $b_3$.

\begin{defn}
We define a  {\em signed norm function} $N:\R^2\longrightarrow \R$ by assigning to
$z= (v, w) \in \R^2$ the value
\begin{align*}
N(z) &:= sign(v,w)\|(v,w)\| \\
&= sign(v,w) \sqrt{v^2+w^2},
\end{align*}
where  $sign(z) = sign(v,w)$ is defined as
\begin{displaymath}
   sign(v,w) = \left\{
     \begin{array}{lr}
       1 & \text{if } w>0 \text{ or } (w=0 \,  \& \, v\geq 0)\\
       -1 & \text{if } w<0 \text{ or } (w=0 \, \& \, v<0)\\
     \end{array}
   \right.
\end{displaymath} 
 If $z$ is viewed as a complex number, then $sign(z)>0$ if and only if $Im(z)>0$ or $z$ is non-negative real.

\end{defn}

Given $a\in [145]$ and $b\in [123]$, we will denote $x = N(a)$ and $y = N(b)$, that is, $x$ and $y$ will be used to denote signed norm of points on $[145]$ and points on $[123]$ respectively.  We will also denote 
\begin{align*}
X_i &:= \{x\in \mathbb{R}^*: x \text{ is signed norm of some point on } [145] \text{ of color i}\}\\
Y_j &:= \{y \in \mathbb{R}^*: y \text{ is signed norm of some point on } [123] \text{ of color j}\}
\end{align*}
for $i = 4,5$ and $j= 2,3$.\\

The following three claims will show that there is some multiplicative structure on the sets $X_i$ and $Y_j$.

From elementary geometry, we know that four points $a,a',b,b'$ are on the same circle if and only if the associated signed norms satisfy $xx' = yy'$. We now define a function $h:(\mathbb{R}^*)^3 \longrightarrow \mathbb{R}^*$ by $h(r_1|r_2r_3) = \frac{r_2r_3}{r_1}$. 

\begin{cl}\label{functionh-lem}
If $x_i \in X_i$, $y_2\in Y_2$ and $y_3\in Y_3$, then $h(x_i|y_2y_3) \in X_i$. Similarly, if $y_j \in Y_j$, $x_4 \in X_4$ and $x_5 \in X_5$, then $h(y_j|x_4x_5) \in Y_j$. (i=4,5 and j = 2,3).
\end{cl}
\begin{proof} (of Claim \ref{functionh-lem})

Let $b_2, b_3$ be the points on $[123]$ with signed norm $y_2, y_3$ and $a_i, a$ be the points on $[145]$ with signed norm $x_i, h(x_i|y_2y_3)$ respectively. Notice that 
\begin{align*}
x_i h(x_i|y_2y_3) &= x_i \frac{y_2y_3}{x_i}\\
&= y_2y_3
\end{align*}
so $a_i,a,b_2,b_3$ are on the same circle. Hence $a$ has color i. This proves that $h(x_i|y_2y_3)\in X_i$. (see Figure~\ref{fig4})

\begin{figure}[h!]
\begin{center}
\includegraphics[width=70mm]{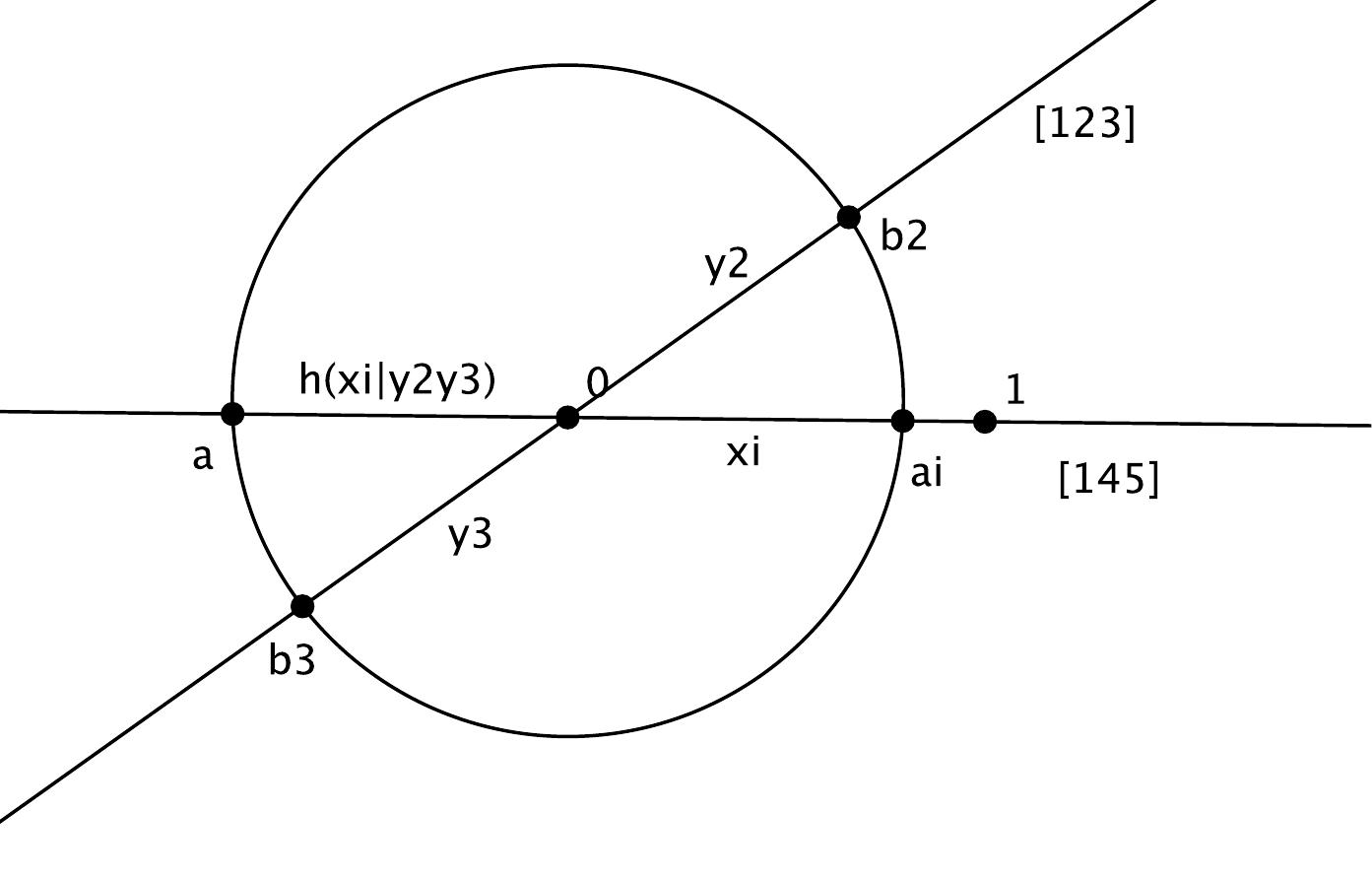}
\end{center}
\caption{}
\label{fig4}
\end{figure}

Similarly, the second part of the claim follows from the same argument.
\end{proof}

We define another function $m:(\mathbb{R}^*)^3 \longrightarrow \mathbb{R}^*$ by $m(r_1|r_2r_3) =\frac{ r_3}{r_2}r_1$. 

\begin{cl}\label{functionm-lem}
If $x_i \in X_i$, $y_j, y_j'\in Y_j$, then $m(x_i|y_jy_j') \in X_i$. Similarly, if $y_j \in Y_j$, $x_i, x_i' \in X_i$, then $m(y_j|x_ix_i') \in Y_j$ ($i=4,5$ and $j = 2,3$).
\end{cl}
\begin{proof} (of Claim \ref{functionm-lem})

Choose $y_{5-j}\in Y_{5-j}$ (notice that $j$ and $5-j$ will be two different numbers in $\{2,3\}$), then 
\begin{align*}
m(x_i|y_jy_j') &= \frac{y_j'}{y_j}x_i = \frac{y_j'y_{5-j}}{y_jy_{5-j}}x_i \\
&= h(h(x_i|y_jy_{5-j})|y_j'y_{5-j})
\end{align*}

By the Claim \ref{functionh-lem}, $h(x_i|y_jy_{5-j})\in X_i$, then so is $h(h(x_i|y_jy_{5-j})|y_j'y_{5-j})$. Therefore, $m(x_i|y_jy_j')\in X_i$.

The second part of the claim follows from the same argument.

\end{proof}

\begin{cl}\label{group-lem}
$X_5$ is a subgroup of $\mathbb{R}^*$, and $Y_2,Y_3,X_4$ are cosets of $X_5$ in $\mathbb{R}^*$. Moreover, the fourth power of each coset representative of $Y_2,Y_3,X_4$ lies in $X_5$.
\end{cl}
\begin{proof} (of Claim \ref{group-lem})

First notice that the point $s_5=(1,0)$ is colored by 5, so we have that $1\in X_5$.

Let $y_j,y_j'\in Y_j$, then by Claim \ref{functionm-lem}, $y_j'/y_j = m(1|y_jy_j') \in X_5$. Hence, we have $Y_j \subset y_j X_5$

Let $y_j\in Y_j$, and $x_5 \in X_5$, then by Claim \ref{functionm-lem}, $x_5y_j = m(y_j|1x_5) \in Y_j$. So we have $y_j X_5 \subset Y_j$. Therefore, $Y_j = y_j X_5$.

Now let $x_5,x_5'\in X_5$, choose $y_j\in Y_j$. Since $Y_j = y_jX_5$,  $x_5y_j, x_5'y_j \in Y_j$, we have $x_5x_5'^{-1} = (x_5y_j)/(x_5'y_j) \in X_5$. So $X_5$ is a subgroup of $\mathbb{R}^*$. We will now denote this group by $G$ and notice that we have also shown that $Y_j = y_j G$.

Now let $x_4\in X_4$, $y_j\in Y_j$ and $g\in G$, then $gy_j \in Y_j$, so by Claim \ref{functionm-lem}, $gx_4 = m(x_4|y_j,gy_j) \in X_4$. So we have $x_4G\subset X_4$.

On the other hand, if $x_4,x_4'\in X_4$, $y_j \in Y_j$, then by Claim \ref{functionm-lem} again, $(x_4'/x_4)y_j = m(y_j|x_4x_4') \in Y_j$, so $x_4'/x_4\in G$. So we have $X_4\subset x_4 G$. Therefore, $X_4 = x_4G$.

To prove the moreover part, let $y_2\in Y_2,y_3\in Y_3$, then $y_2y_3 = h(1|y_2y_3) \in G$. Let $x_4 \in X_4$, then $y_2y_3/x_4 = h(x_4|y_2y_3) \in x_4G$. So we have $x_4^2 \in G$.

Similarly, let $x_4\in X_4, y_j \in Y_j$, $x_4/y_j = h(y_j|1,x_4) \in y_jG$, so we have $y_j^2 \in x_4G$, hence $y_j^4 \in G$. Therefore, the claim follows.
\end{proof}

The Claims \ref{functionh-lem} to \ref{group-lem} do  not suffice by themselves
to get a contradiction, see Proposition \ref{pr53}. We need a further argument. The following three claims show that the colorings of the lines passing through the origin are strictly restricted due to the multiplicative structure.

\begin{cl}\label{3-color-line-lem}
For every line $l$ through origin, $|\Gamma(l)| = 3$.
\end{cl}
\begin{proof} (of Claim \ref{3-color-line-lem})

By Claim  \ref{two-sided-lem}, we let $a_4,a_5$ be of color 4,5 on $[145]$ and on opposite side of $\vec 0$, and $b_2,b_3$ be of color 2,3 on $[123]$ and on opposite side of $\vec 0$. Given any line $l$ through $0$, the circle $(b_2b_3a_4)$ intersects $l$ at two points. The intersection points must be of color 2 or 3 or 4. Without loss of generality, we assume the intersection points are of color 2, then the circle $(b_3a_4a_5)$ intersects $l$ at two other points. The intersections are of color 3,4 or 5. Therefore, the claim follows (see Figure~\ref{fig5}).

\begin{figure}[h!]
\begin{center}
\includegraphics[width=70mm]{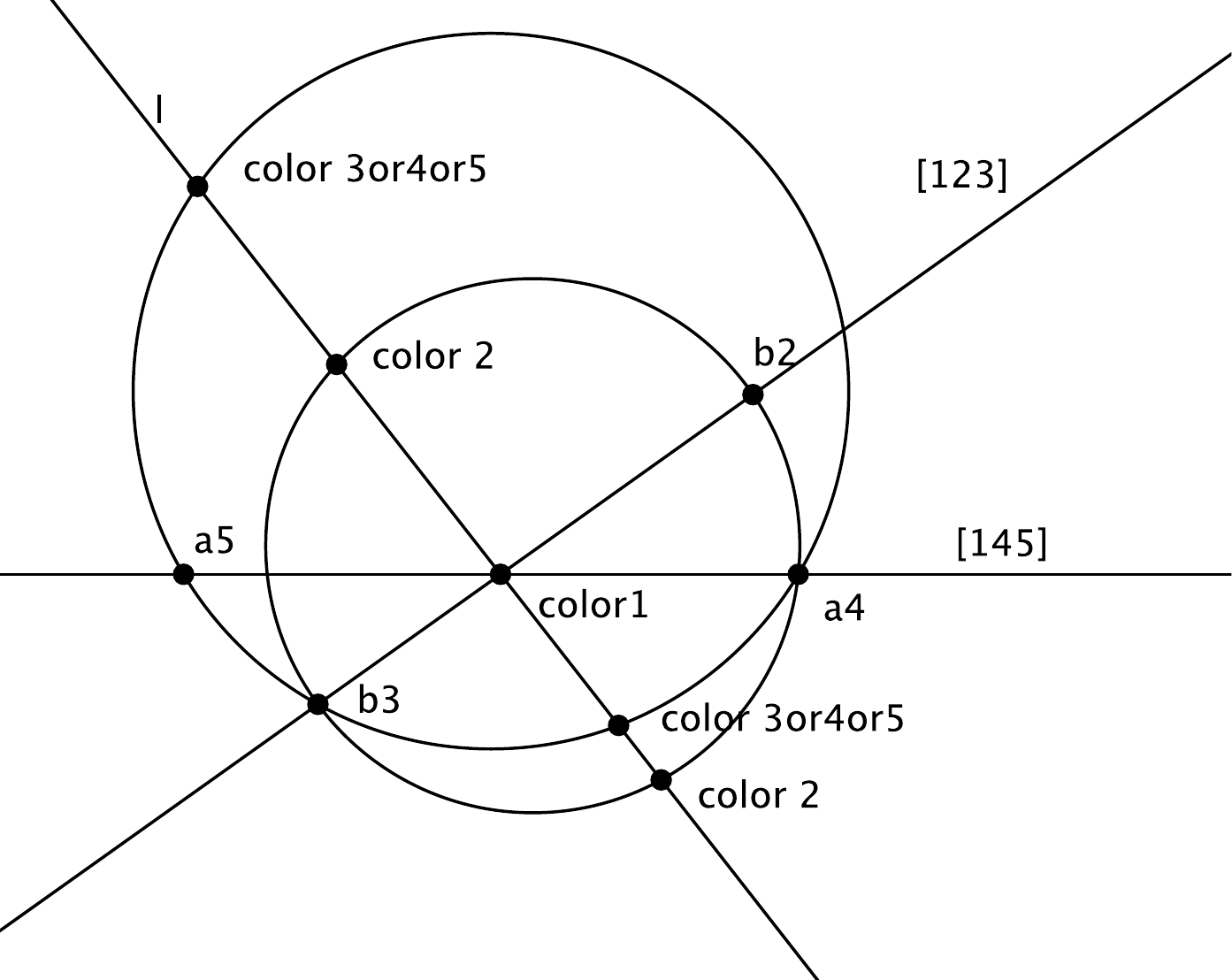}
\end{center}
\caption{}
\label{fig5}
\end{figure}
\end{proof}

\begin{cl}\label{distinct-color-lem}
Let $C$ be a circle intersecting $[123]$ and $[145]$ transversally. For $i = 4,5$, if $\Gamma(C\cap [145]) = \{i\}$ then $\Gamma(C\cap [123]) = \{1\}$ or $\Gamma(C\cap [123]) = \{2,3\}$.

Similarly, for $j = 2,3$ if $\Gamma(C\cap [123]) = \{j\}$, then $\Gamma(C\cap [145]) = \{1\}$ or $\Gamma(C\cap [145]) = \{4,5\}$.

\end{cl}

\begin{proof} (of Claim \ref{distinct-color-lem})

Let $\Gamma(C\cap [145]) = \{i\}$ and assume that $\Gamma(C\cap [123]) \neq \{1\}$, without loss of generality, assume that $2\in \Gamma(C\cap [123])$. Let $C\cap [145] = \{a_i, a_i'\}$ and $C\cap [123] = \{b_2, b\}$. Let $x_i,x_i', y_2, y$ be the signed norms associated to $a_i,a_i',b_2, b$, then $x_ix_i' = y_2y$. Notice that for either case $i=4$ or $i=5$, $x_ix_i' \in G$, so we have $y = g/y_2$ for some $g\in G$. Choose $y_3\in Y_3$, then $y_2y_3 = h(1|y_2y_3)\in G$, so we have $y = y_3 g'$ for some $g'\in G$. Therefore, $y\in Y_3$, which means $b$ is of color 3 (see Figure~\ref{fig6a}), so $\Gamma(C\cap [123]) = \{2,3\}$.

Let $\Gamma(C\cap [123]) = \{j\}$ and assume that $\Gamma(C\cap [145]) \neq \{1\}$, without loss of generality, assume that $4\in \Gamma(C\cap [145])$. Let $C\cap [123] = \{b_j, b_j'\}$ and $C\cap [145] = \{a_4, a\}$. Let $y_j,y_j', x_4, x$ be the signed norms associated to $b_j,b_j',a_4, a$, then $xx_4 = y_jy_j'$. As in the proof claim \ref{group-lem}, $y_jy_j' \in X_4$, so we have $x = y_jy_j'/x_4 \in G$. Therefore, $a$ is of color 5 (see Figure~\ref{fig6b}), so $\Gamma(C\cap [145]) = \{4,5\}$.

The other cases follow from exactly the same arguments. 
\begin{figure}[h!]
\begin{center}
\begin{subfigure}[h!]{0.4\textwidth}
\centering
\includegraphics[width=50mm]{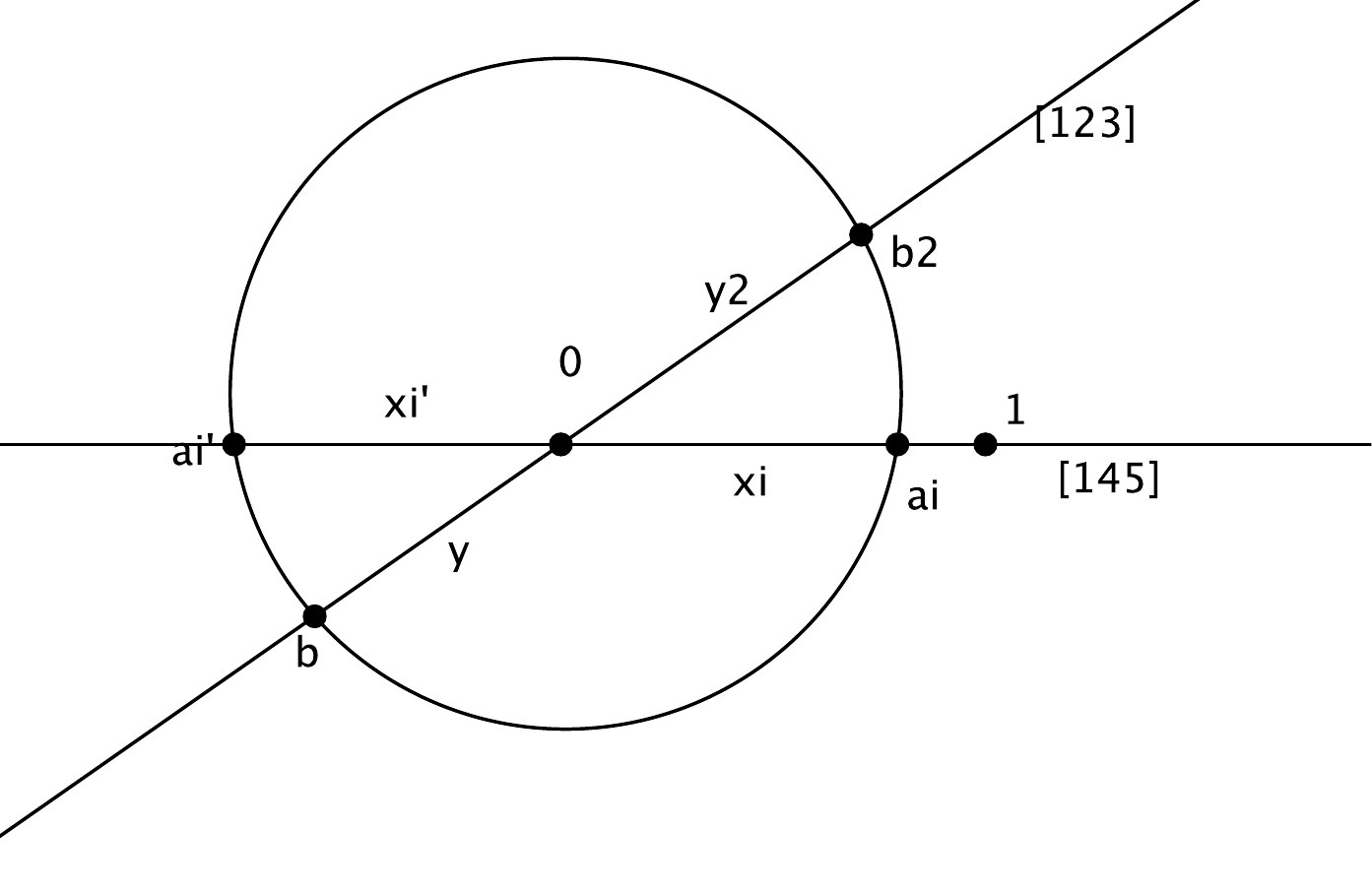}
\caption{}
\label{fig6a}
\end{subfigure}
\begin{subfigure}[h!]{0.4\textwidth}
\centering
\includegraphics[width=50mm]{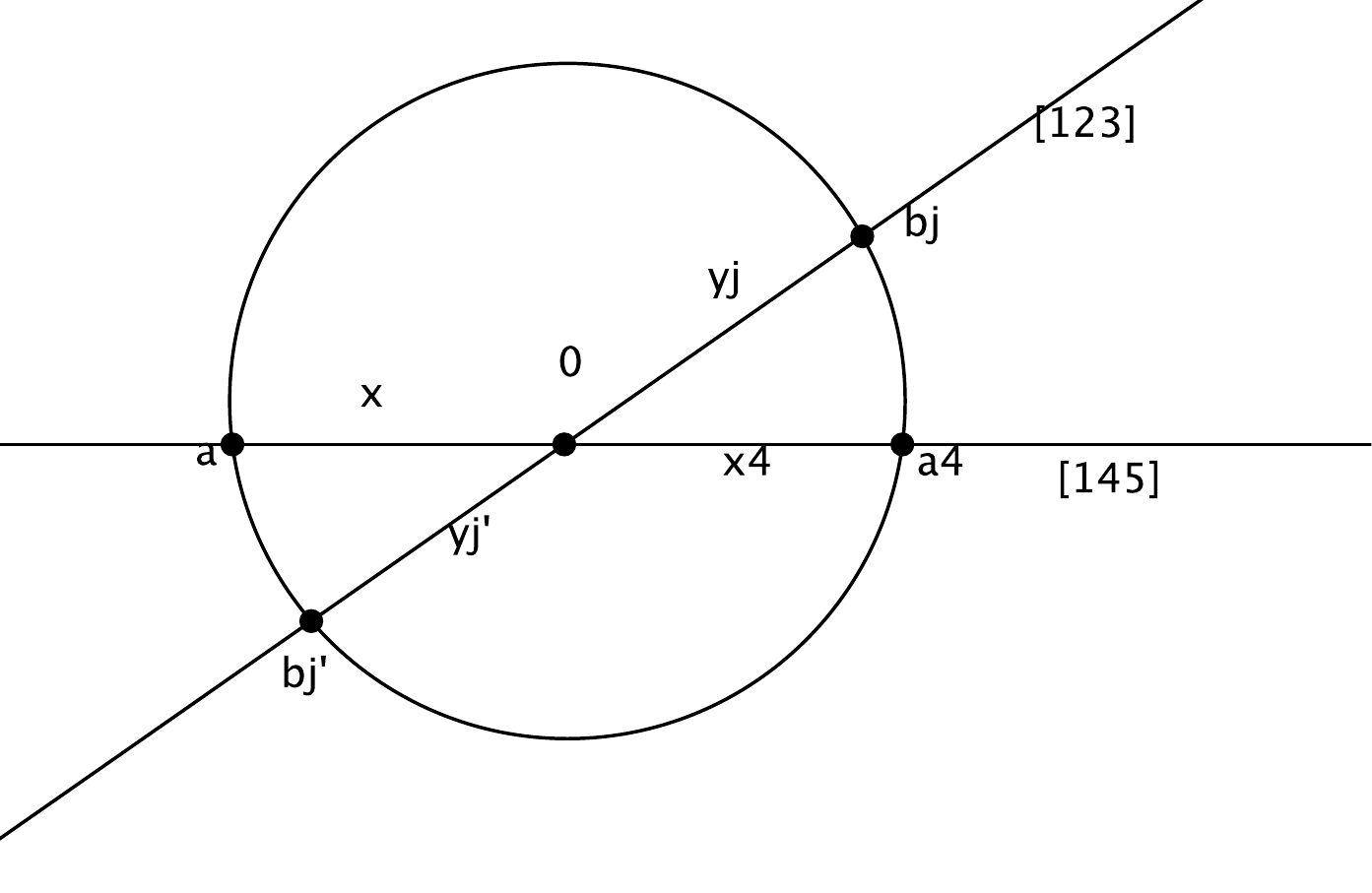}
\caption{}
\label{fig6b}
\end{subfigure}
\end{center}
\caption{}
\label{fig6}
\end{figure}
\end{proof}

\begin{cl}\label{intersect-lem}
For each line $l$ through origin, $\Gamma(l) = \{1,2,3\}$ or $\Gamma(l) = \{1,4,5\}$.
\end{cl}
\begin{proof} (of Claim \ref{intersect-lem})

We assume there is a line through $\vec 0$ of color 1,2,4. Denote this line by $[124]$. Choose points $a_4,a_5$ of color 4,5 on the opposite side of $\vec 0$ on $[145]$, and $b_3$ of color 3 on $[123]$. The circle $(b_3a_4a_5)$ intersects $[124]$ at two points of color 4. The two intersection points are on the opposite side of $\vec 0$, so we may assume that there is a pair of points of color 2 and 4 on the opposite side of $\vec 0$.
Let $c_2, c_4$ be of color 2,4 on the opposite side of $0$ on $[124]$, the circle $(234):=(c_2b_3c_4)$ intersects $[145]$ at two points $a_4,a_4'$ of color 4 (see Figure~\ref{fig7a}). 

Under a \m transformation $T$, we may map $a_4$ to $\vec 0$, $a_4'$ to $\infty$ and $\vec0$ to $(1,0)$. Under this transformation, the extended line$(145)$ is mapped to the extended X-axis, the circle $(234)$ is mapped to an extended line through origin, and the extended line $(124) = [124]\cup\{\infty\}$ is mapped to a Euclidean circle (see Figure~\ref{fig7b}). After the transformation, the Euclidean circle $(124)$ intersect $[145]=(145)-\{\infty\}$ and $[234]=(234)-\{\infty\}$ on four points. Then by exchanging the role of color $1$ and color 4, exact same argument shows that all previous claims holds. But the circle (124) intersects [145] and  [234] transversally and $\Gamma((124)\cap [145]) = \{1\}$, $\Gamma((124)\cap [234]) = \{2,4\}$, which is a contradiction to Claim  \ref{distinct-color-lem} (by exchanging the role of color 1 and 4 in the statement).

\begin{figure}[h!]
\begin{center}
\begin{subfigure}[h!]{0.4\textwidth}
\centering
\includegraphics[width=50mm]{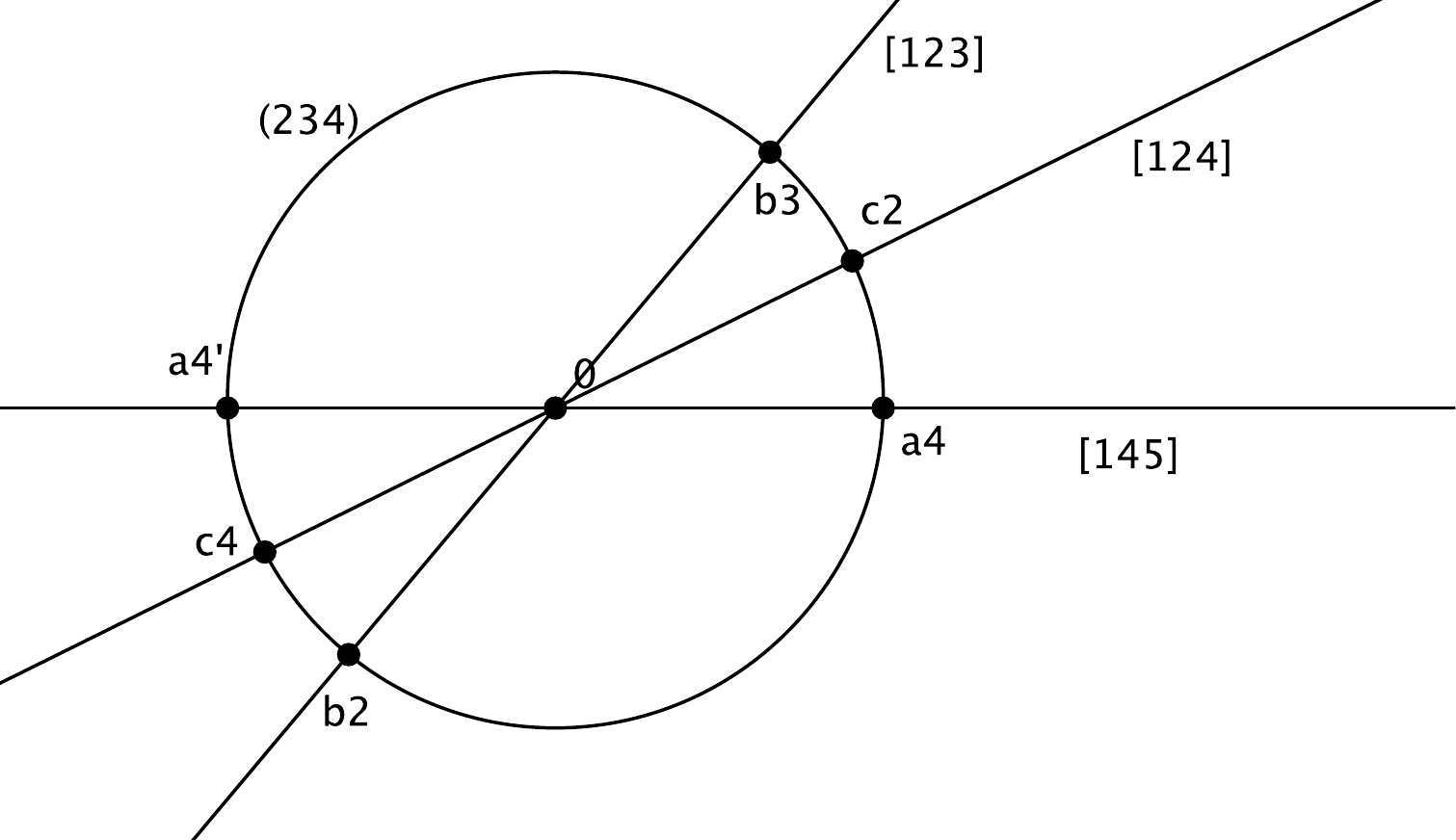}
\caption{Before Transformation}
\label{fig7a}
\end{subfigure}
\begin{subfigure}[h!]{0.4\textwidth}
\centering
\includegraphics[width=50mm]{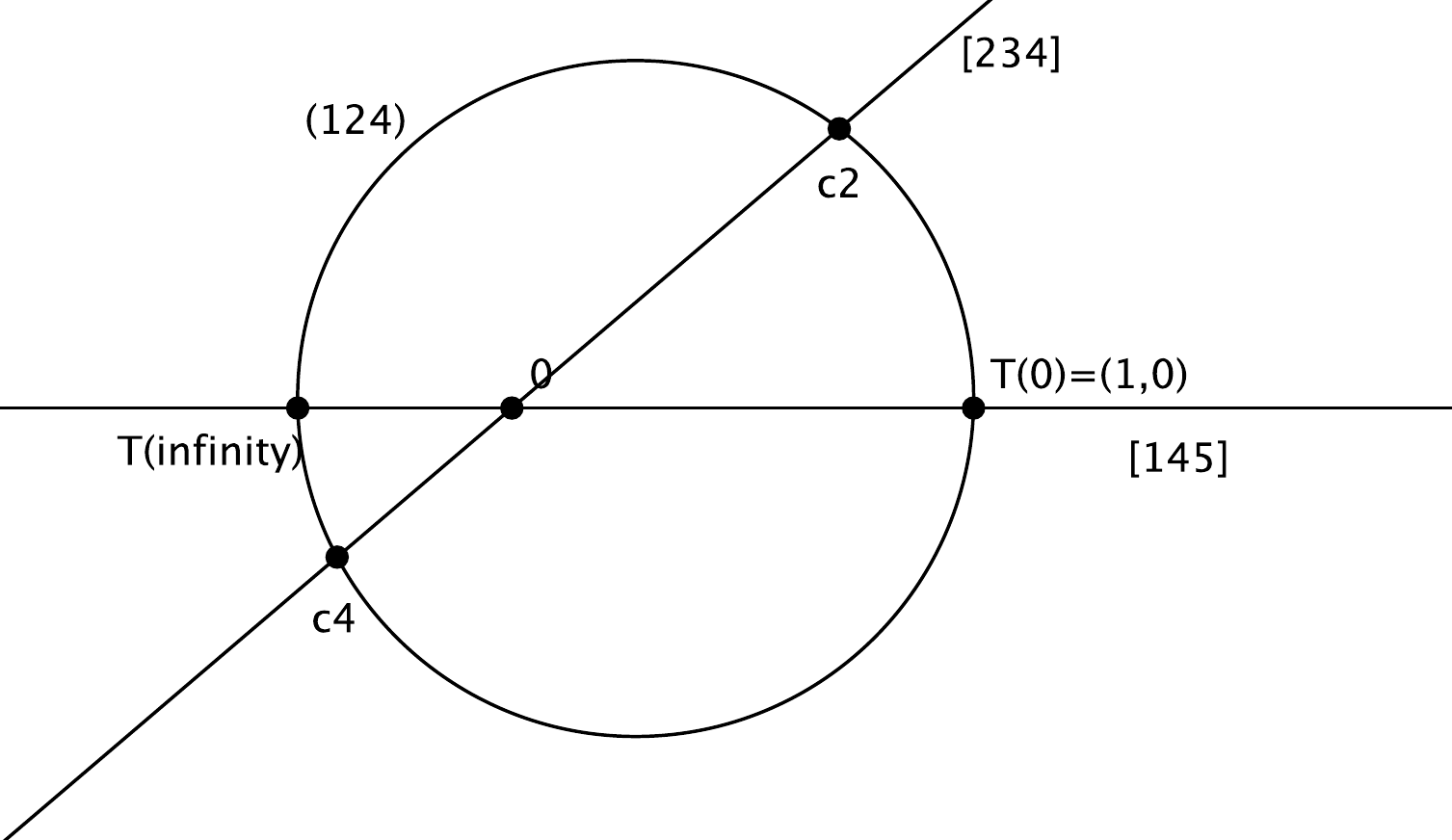}
\caption{After Transformation}
\label{fig7b}
\end{subfigure}
\end{center}
\caption{}
\label{fig7}
\end{figure}

Similarly, the line through $\vec 0$ can not be of colors 1,2,5 or 1,3,4 or 1,3,5. So together with Claim \ref{3-color-line-lem}, we conclude that each line through the origin has coloring either 1,2,3 or 1,4,5.
\end{proof}

(Return to the proof of Theorem \ref{th31})
Now let $[123]'$ be another line through 0 with color 1,2,3, then the exact argument shows that the signed norms of points of color 2,3 on $[123]'$ form cosets of group $G$ in $\mathbb{R}^*$ and the fourth power of the representatives are in $G$. If $[145]'$ be another line through 0 with color 1,4,5, choose $y_2\in Y_2$, then under the \m transformation $\phi(z)=(1/y_2)z$, the signed norms of points color 2 on $[123]$ then form the same group $G$. The exact same argument shows that the signed norms of points of color $4,5$ on $[145]'$ after the transformation form cosets of group $G$ in $\mathbb{R}^*$ and the fourth power of the representatives are in $G$. If we apply $\phi^{-1}$, any signed norm set is multiplied by a factor of $y_2$. The resulting signed norm set of points of color $4$ or $5$ on $[145]'$ is again a coset of $G$ and the fourth power of the representative is again in $G$.

Now let $H := \cup_{x^4\in G} xG$, then $H$ is a subgroup of $\R^*$. Given $z\in\R^2$ of color 2, 3, 4 or 5, by the result in the above paragraph, the signed norm $N(z) \in H$, then so is usual norm $\|z\|$ as $-1\in H$. Therefore, $\{\|z\|: z\in \R^2, z \text{ is of color 2, 3, 4 or 5}\}$ is a subset of $H$. Notice we can find a Euclidean circle which does not contain color 1 and whose center is not $\vec 0$, so $\{\|z\|: z\in \R^2, z \text{ is colored 2,3,4,5 }\}$  contains an non-degenerate interval. This implies $H$ contains a non-degenerate interval. Hence $\mathbb{R}^+ \subset H$ as $H$ is a subgroup of $\mathbb{R}^*$.

Now given $r\in \mathbb{R}^+$, then $r^{1/4} \in \mathbb{R}^+$. so $r^{1/4} \in H$. Hence, we write $r^{1/4} = xg$ with $x^4\in G$, then $r = (r^{1/4})^4 = (xg)^4 = x^4g^4 \in G$. Therefore, we have $\mathbb{R}^+ \subset G$. By Claim \ref{two-sided-lem}, $G$ contains at least one  negative value, hence $G = \mathbb{R}^*$. But this is a contradiction as this means there is no color 4 on the line $[145]$. This proves Theroem \ref{th31}.
\end{proof}

\begin{rmk}
If we identify $\R^2_\infty$ with the Riemann surface $\hat{\mathbb{C}}=\mathbb{P}^1( \mathbb{C} )$,
then  it is well known that four points $z_1,z_2,z_3,z_4$ lie on the same circle in $\hat{\mathbb{C}}$ if and only if their cross ratio 
$[z_1,z_2:z_3,z_4] = \frac{(z_1-z_3)(z_2-z_4)}{(z_2-z_3)(z_1-z_4)}$ is real. Using this result, one can
give an alternative derivation of  the  multiplicative structure of  the coloring sets  given in  Claim \ref{group-lem}.
\end{rmk}

\section{Sharpness of Main Result}\label{sec6}
 
The following two propositions give two examples  showing different ways in
which  the main result  Theorem \ref{th11b} is sharp.

The first result shows that the conclusion of Theorem \ref{th11b} fails if the hypothesis is weakened
to replace full $(n+3)$-coloring by full $(n+2)$-coloring.
\begin{prop}\label{pr12}
There is a full $(n+2)$-coloring of ${\mathbb S}^n$ in which all $(n-1)$-spheres are colored using at
most $(n+1)$ colors.
\end{prop}
\begin{proof}
Again, we identify $\mathbb{S}^n$ with $\R^n_\infty$.
Denote $\R^i :=\{(x_1,...,x_n)\in\R^n:x_{i+1}=x_{i+2}=...=x_n = 0\}$ Consider the following partition,
\begin{align*} 
\Gamma_1 &= \{\vec 0\},\\
\Gamma_2 &= \{\infty\}, \\
\Gamma_i &= \R^{i-2} - \R^{i-3} \text{ for }i\geq 3
\end{align*}

It is clear that $\cup_{i=1}^{n+2}\Gamma_i = \mathbb{R}^n\cup\{\infty\}$, and $\Gamma_i\cap\Gamma_j = \emptyset$ if $i\neq j$. So $\Gamma=\{\Gamma_i:i=1,..,n+2\}$ is indeed a full $(n+2)$-coloring of $\mathbb{S}^n$. 

Suppose for contradiction that there is an $(n+2)$-colored $(n-1)$-sphere $S$. Since $S$ has color 1 and 2, $\vec 0\in S$ and $\infty \in S$, so $S$ is an extended $(n-1)$-plane through the origin.

We will prove by induction that $\R^k$ for all $k$. For the base case, since $S$ contains color $3$, $S\cap (\R^1-\R^0) \neq \emptyset$, so $\R^1 \subset S$. For the induction step, assume that $\R^{k-1} \subset S$. Since $S$ has color $k+2$, $S\cap (\R^k-\R^{k-1} \neq \emptyset$, hence $\R^k \subset S$, completing the induction step. 

Therefore, by induction, we have $\R^n \subset S$, which is a contradiction.
\end{proof}

The second result shows that no matter how many color classes we impose on
the coloring, the conclusion of Theorem \ref{th11b} cannot be strengthened to guarantee
the existence of  an
$(n-1)$-sphere requiring $(n+3)$ or more colors.

\begin{prop}\label{pr52}
For each $k\geq n+3$, there is a full  $k$-coloring of $\mathbb{S}^n$ in which all  $(n-1)$-spheres
are colored with  $(n+2)$ or fewer colors.
\end{prop}
\begin{proof}
Let $M=\{x_1,...,x_{k-1}\}\subset \mathbb{S}^n$ of $k-1$ points with no $(n+2)$ points in $M$ are in the same $\mathbb{S}^{n-1}$. This is possible to construct by induction. Now if we let $\Gamma_i = \{x_i\}$ and $\Gamma_k = \mathbb{S}^n-M$, then there is no $(n+3)$-colored $(n-1)$-sphere in this coloring of $\mathbb{S}^n$.
\end{proof}

The following example gives some indication that the proof of the main theorem for dimension 2 cannot be substantially simplified.


\begin{prop}\label{pr53}
Given $C_1$, $C_2$ circles in $\mathbb{S}^2$ with $C_1$ intersecting $C_2$ transversally. There is a 5-coloring $\Gamma$ of $C_1\cup C_2$ such that for every circle $C$ in $\mathbb{S}^2$, $|\Gamma(C\cap (C_1\cup C_2))|\leq 3$.
\end{prop}

\begin{rmk}
In the proof of Theorem \ref{th31} in dimension 2, we defined the notion of signed norm and shown that the signed norm sets on two intersecting lines for color 2,3,4,5 are actually cosets of a subgroup in $\R^*$. Proposition \ref{pr53}
 shows that the constraints these impose of a configuration of two intersecting lines are not enough to lead a contradiction.
\end{rmk}
\begin{proof}
Under a \m transformation, we may map one of the intersection points to $\infty$ and the other to $(0,0)$. Then $C_1$ and $C_2$ are mapped to two extended lines through origin. We define a 5-coloring as follows:

\begin{align*}
\Gamma_2&:=\{\vec x\in C_2: \|\vec x\| \in 2^{1/4}\mathbb{Q}^*\}\\
\Gamma_3&:=\{\vec x\in C_2: \|\vec x\| \in 2^{-1/4}\mathbb{Q}^*\}\\
\Gamma_4&:=\{\vec x\in C_1: \|\vec x\| \in 2^{1/2}\mathbb{Q}^*\}\\
\Gamma_5&:=\{\vec x\in C_1: \|\vec x\| \in \mathbb{Q}^*\}\\
\Gamma_1&:=C_1\cup C_2-\cup_{i=2}^5\Gamma_i
\end{align*}

Suppose for contradiction that there is a circle $C$ such that $C\cap (C_1\cup C_2)$ contains 4 or more colors. Notice if $C$ is an extended line, then $|C\cap (C_1\cup C_2)|\leq 3$. Therefore, $C$ is a Euclidean circle and intersects each $C_1$ and $C_2$ at 2 points. Denote $a,b$ be intersection points on $C_1$ and $c,d$ be intersection points on $C_2$. By assumption, $a,b,c,d$ are of different colors, notice three of them must be of color 2,3,4 or 5. Hence, there are 4 cases to consider, we will prove one case and the other cases can be treated similarly.

Suppose this circle $C$ contains color 2,4 and 5. We let $a, b, c$ are of color 4, 5 and 2 respectively. By elementary geometry, we have $\|a\|\|b\| = \|c\|\|d\|$. A straight forward calculation shows that $\|d\|\in 2^{1/4}\mathbb{Q}^*$, so $d$ has color 2, which is a contradiction.

Similar arguments and calculations show that the other cases are not possible as well. Therefore, the proposition follows.
\end{proof}
\section{Analogue of Main Theorem in Euclidean Geometry}\label{euclid}
In section \ref{main-thm-euclid}, we introduced the Euclidean analogue of the main result where $(n-1)$-spheres are replaced by great $(n-1)$-spheres. Similar to the main theorem, to get a great $(n-1)$-sphere of $n$ different colors is immediate as we can pick $n$ points on $\mathbb{S}^n$ of distinct colors and form a $n$ dimensional vector space that contains these points, then the intersection of the vector space and $\mathbb{S}^n$ will be a great $(n-1)$-sphere. The content of the theorem lies wholly in replacing $n$ by $n+1$. However, in contrast to Theorem \ref{th11b}, the proof in this case is quite simple. The key observation is the following.\\

{\bf Observation:} Any great $(n-1)$-sphere and great circle (i.e., great $1$-sphere) on $\mathbb{S}^n$ must intersect.
\begin{proof}
Given a great $(n-1)$-sphere $S$ and a great circle $C$ on $\mathbb{S}^n$, let $P$ be the hyperplane through the origin such that $S = P\cap \mathbb{S}^n$ and $Q$ be a two plane through the origin such that $C = Q\cap \mathbb{S}^n$. Notice $P$ is a vector space in $\R^{n+1}$ of dimension $n$ and $Q$ is a vector space in $\R^{n+1}$ of dimension $2$, so $P\cap Q$ is a vector space in $\R^{n+1}$ of dimension at least $1$. Therefore, $S\cap C = P\cap Q\cap \mathbb{S}^n$ contains at least two points.
\end{proof}

\begin{proof}[Proof of Theorem \ref {great-sphere-thm}.]
It  is enough to prove the theorem for $k= n+2$.

We will argue by contradiction. Suppose that there is no great $(n-1)$-sphere that contains $(n+1)$ or more different colors, then there is no great circle that contains at least $3$ different colors either (as we can always build a great $(n-1)$-sphere of $(n+1)$ colors from a great circle of $3$ colors). Now we let $S$ be a great $(n-1)$-sphere of color $1, 2,..., n$ and $C$ be a great circle of color $n+1$ and $n+2$. Notice that by assumption, $S$ only has points of color $1, 2,..., n$ and $C$ only has points of color $n+1$ and $n+2$. However, by the observation, $S\cap C \neq \emptyset$. This is a contradiction.
\end{proof}

Similar to the main result, we have the following two propositions showing different ways in which this result is sharp.

The first result shows that the conclusion of Theorem \ref {great-sphere-thm} fails if the hypothesis is weakened to replace full $(n+2)$ coloring by full $(n+1)$-coloring.
\begin{prop}
There is a full $(n+1)$-coloring of $\mathbb{S}^n$ in which all great $(n-1)$-spheres are colored using at most $n$ colors.
\end{prop}
\begin{proof}
The construction is very similar to the construction in Proposition \ref{pr12}.
We consider $\mathbb{S}^n$ as a subset of $\R^{n+1}$. 
Denote $\R^i :=\{(x_1,...,x_{n+1})\in\R^{n+1}:x_{i+1}=x_{i+2}=...=x_{n+1} = 0\}$. Consider the following partition:
\begin{equation*}
\Gamma_i = (\R^i-\R^{i-1})\cap \mathbb{S}^n,  ~~~ 1 \le i \le n+1.
\end{equation*}
It is clear that $\Gamma := \{\Gamma_i:i=1,...,n+1\}$ is indeed a full $(n+1)$-coloring of $\mathbb{S}^n$.
\begin{figure}[h!]
\begin{center}
\includegraphics[width=70mm]{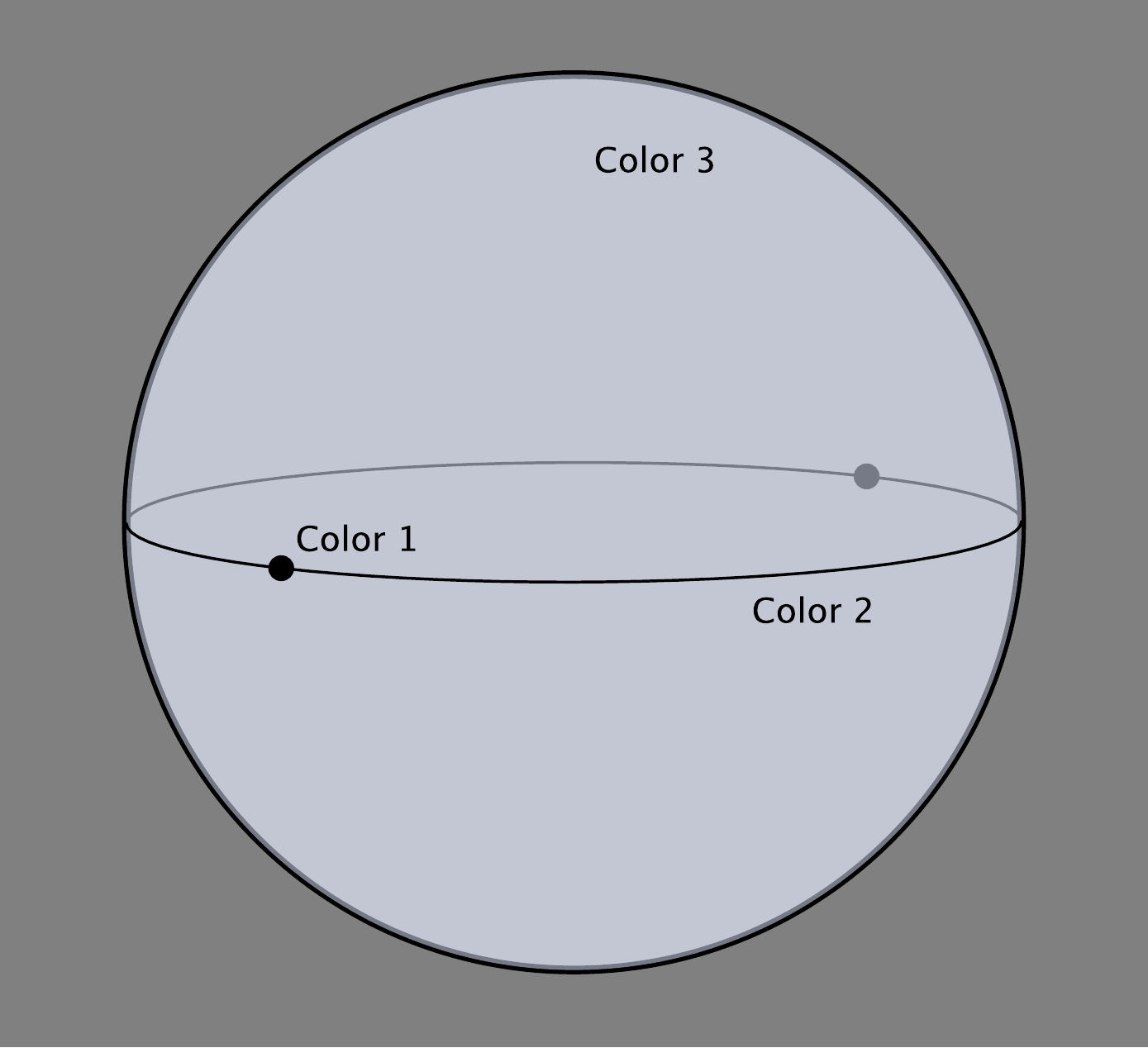}
\end{center}
\caption{Coloring of $\mathbb{S}^2$}
\label{fig8}
\end{figure}

We argue by contradiction to show that there is no great $(n-1)$-sphere containing $(n+1)$ colors. Suppose that there
were, and let $S$ be   a great $(n-1)$ sphere $S$ that contains all the colors.
Then let $P$ be a hyperplane in $\R^{n+1}$ such that $S = P\cap \mathbb{S}^n$. 

We will prove by induction on $k \ge 1$  that $\R^k \subset P$ for all $k$.
For the base case, since $S$ has color $1$, we have $\R^1 \subset P$.
For the induction step, assume that $\R^{k-1} \subset P$. Since $S$ has color $k$, 
$P\cap (\R^k-\R^{k-1}) \neq \emptyset$, hence $\R^k \subset P$, completing the induction step.
Therefore, by induction, we have $\R^{n+1} \subset P$ which is a contradiction.
\end{proof}

Figure~\ref{fig8} shows the partition in $\mathbb{S}^2$. The two antipodal points are colored by color 1, the great circle passing through these two points is colored by color 2 (except at the two points) and the rest of the sphere is colored by color 3.

The second result shows that no matter how many color classes we impose on the coloring, the conclusion of Theorem  \ref {great-sphere-thm} cannot be strengthened to guarantee the existence of a great $(n-1)$-sphere of $(n+2)$ or more colors.

\begin{prop}
For each $k\geq n+2$, there is a full $k$-coloring of $\mathbb{S}^n$ in which all great $(n-1)$-spheres are colored with $(n+1)$ or fewer colors.
\end{prop}
\begin{proof}
Let $M = \{x_1,...,x_{k-1}\}\subset \mathbb{S}^n \subset \R^{n+1}$ consists of $k-1 \ge n+1$ points 
such that  every subset of $n+1$ points in $M$ are linearly independent. 
Thus there is no great $(n-1)$-sphere containing $(n+1)$ or more points in $M$.
Now if let $\Gamma_i = \{x_i\}$ and $\Gamma_k = \mathbb{S}^n - M$, then 
there is no great $(n-1)$-sphere containing $(n+2)$ or more different colors.
\end{proof}

\section{Application: Characterizing \m Transformations in the Plane}\label{sec4} 

We now prove the ``five-point theorem".

\begin{proof}[Proof of Theorem \ref{5point}.]

Let $T:\mathbb{S}^2\longrightarrow \mathbb{S}^n$ be a weakly circle preserving map.

If $T(\mathbb{S}^2)$ contains at least $6$-points, then by the six-point theorem above, the image 
$T(\mathbb{S}^2)$ is a $2$-sphere and $T$ must be a \m transformation.

It remains to deal with the case in which the image $T(\mathbb{S}^2)$ contains exactly $5$ points.
 We will show that no  such transformation exists. Suppose for contradiction that such a map $T$ existed and denote its image $T(\mathbb{S}^2) =\{m_1,m_2,m_3,m_4,m_5\}$. Define a coloring of $\mathbb{S}^2$ by the partition $\Gamma = \{\Gamma_i:i=1,2,...,5\}$ with $\Gamma_i = T^{-1}(m_i)$ for $1\leq m \leq 5$. 
 By Theorem \ref{th11b}, there will exist some circle in the domain $\mathbb{S}^2$ that requires $4$ or more colors. 
 Under $T$, the image of this circle contains $4$ points in $T(\mathbb{S}^2)$, 
 and since the map is weakly circle-preserving, these  $4$ points lie in on a circle
 in the range space. This  contradicts the hypothesis  that the five points comprising the image $T(\mathbb{S}^2)$ 
 are  in circular general position.
\end{proof}

\begin{rmk} 
It seems likely that the method of proof of the ``five-point theorem" given here can be  generalized to yield a local 
``five-point theorem" exactly  parallel
to the ``six point theorem" in \cite[Theorem 1]{GW79}. That is, the map can be taken to be $ T: U \to \mathbb{S}^n$,
with $U \subset \mathbb{S}^2$, where $U$ is a simply connected open set.
\end{rmk}

The following proposition shows  Theorem \ref{5point} is sharp in the sense that  if we change the hypothesis to 
$4$ points, then the weakly-circle preserving map need not be a \m transformation.

%
\begin{prop} \label{sharp-map}
There exists a weakly circle preserving map $T:\mathbb{S}^2 \longrightarrow \mathbb{S}^n$ such that
its range $T(\mathbb{S}^2)$ 
consists of $4$ points that do not all lie in a circle.
\end{prop}
\begin{proof}
Let $M=\{m_1,m_2,m_3,m_4\}\subset \mathbb{S}^n$ be such that $M$ does not lie in a circle. By Proposition \ref{pr52}, there is a 4-coloring $\Gamma$ on $\mathbb{S}^2$ such that there is no circle containing all of the 4 colors. Denote $\Gamma_i$ as the set of points of color $i$ in $\mathbb{S}^2$, and define $T:\mathbb{S}^2 \longrightarrow \mathbb{S}^n$ by $T(x) = m_i$ if $x\in \Gamma_i$. This map is weakly circle preserving as every circle is mapped to at most three points in $\mathbb{S}^n$. The proposition follows.
\end{proof}

\section{Acknowledgments}
The first author thanks N. Nygaard (University of Chicago) for helpful discussions. 
The second author worked on this project in an 
REU Summer Program at the University of Michigan with supervisor  J.  C. Lagarias. 
Both authors thank J. C. Lagarias for helpful discussions and for editorial assistance.
We acknowledge use of  the computer package {\em Cinderella 2} \cite{RK12}  to draw the figures in this paper.

\end{document}